\documentclass[leqno,11pt]{amsart}
\usepackage{amsfonts, amssymb, amsmath, amsthm, hyperref, color, float, enumerate, enumitem}
\usepackage{mathrsfs}
\usepackage{url}
\usepackage{graphicx}        
\usepackage[bottom]{footmisc}
\usepackage{esint}

\usepackage{newtxtext}       %
\usepackage[varvw]{newtxmath}   
\makeatletter




\makeindex             

\usepackage{scalerel}[2014/03/10]
\usepackage[usestackEOL]{stackengine}
\def\avgint{\,\ThisStyle{\ensurestackMath{%
    \stackinset{c}{0\LMpt}{c}{0\LMpt}{\SavedStyle-}{\SavedStyle\phantom{\int}}}%
    \setbox0=\hbox{$\SavedStyle\int\,$}\kern-\wd0}\int} 
  
 \theoremstyle{plain}
   \newtheorem{teo}{Theorem}
   
   \newtheorem{lemma}{Lemma}

   \newtheorem{defi}{Definition}
\theoremstyle{remark}
 \newtheorem{remark}{Remark}

\begin{document}

\title{On Tent Spaces for the Gaussian Measure}
\author{Liliana Forzani}
\address{Facultad de Ingeniería Química and CONICET, Santiago del Estero 2829, Santa Fe, Argentina.}
\email{liliana.forzani@gmail.com}
\author{Roberto Scotto}
\address{Facultad de Ingeniería Química, Santiago del Estero 2829, Santa Fe, Argentina.}
 \email{roberto.scotto@gmail.com} 
\author{Wilfredo Urbina}
\address{Department of Mathematics and Actuarial Sciences, Roosevelt University, 430 S Michigan Ave. Chicago, IL 60605, USA. }
\email{wurbinaromero@roosevelt.edu} 
%
%



\begin{abstract} Following the scheme of tent spaces in harmonic analysis developed by R. Coifman, Y. Meyer, and E. Stein in \cite{cms}, we succeed in doing so for the Gaussian setting. In \cite{MNP}, part of this theory (an atomic decomposition) is developed for a specific tent space where functions are defined just in a proper subset of $\mathbb{R}^{n+1}_+,$ and without the use of an area function. In the present paper, using a variation of the area function considered in \cite{FSU}, we define the Gaussian area function and Gaussian tent spaces and prove both their atomic decompositions and the characterization of their dual spaces. Some applications are also considered.
\end{abstract}
 \maketitle

\section{Introduction}
\label{sec:1}

In 1985, R. Coifman, Y. Meyer, and E. Stein in their paper {\em Some New Function Spaces and Their Applications to Harmonic Analysis} \cite{cms} introduced the (classical)  tent spaces as a class of functions especially adapted to the study of singular integral operators. The theory of tent spaces has become a fundamental tool in harmonic analysis, both for the study of singular integrals and for Hardy spaces. In their seminal work, the authors developed a rich atomic theory in the Euclidean case, providing insights into the structure and a dual characterization of these function spaces.

In 2012, J. Maas, J. Van Nerveen and P. Portal in \cite{MNP}, 
introduced local Gaussian tent spaces, that is, tent spaces for the Gaussian measure whose functions are defined only on a proper subset of $\mathbb{R}^{n+1}_+$,  
and proved an atomic decomposition for elements of some of these local tent spaces. 
In fact, their approach was restricted to functions defined in the truncated domain \( D = \{(x,t)\in \mathbb{R}^{n+1}_+ : t < \min (1, 1/|x|)\} \), relying on ad hoc constructions rather than the explicit use of an area function. 
Due to the nature of the Gaussian measure, they introduced a sui generis concept of the Gaussian Whitney set, where the Whitney covering lemma can be applied in this setting.

However, the definition of an area function for the Gaussian measure has been
somewhat problematic, due to the lack of a good definition of
\emph{cone region} in this context. In 1994, L. Forzani and E.
Fabes, see \cite{for} and also \cite{FSU}, a Gaussian area function was defined by choosing a {\em pencil-type zone} as a possible {\em Gaussian cone}. Furthermore, the $L^p(\gamma)$ boundedness of that operator was proved.

In this paper, we define an area function equivalent to the one in \cite{FSU}, and use it to define Gaussian tent spaces that resemble the classical harmonic approach developed in \cite{cms}. With that, we lift the support restriction imposed in \cite{MNP}.

We also introduce Gaussian Carleson measures, which are essential in describing the dual of the endpoint tent space \( T^{1,\infty}_{\alpha, \beta}(\gamma) \), and establish duality results for general spaces \( T^{p,q}_{\alpha, \beta}(\gamma) \).  
 
The paper is organized as follows. In Section \ref{prelim}, we introduce the preliminary definitions and properties of Gaussian tent spaces, including the definitions of cutoff function, admissible ball, Gaussian cone and Gaussian tent, area function, atom and Carleson measure. Section \ref{main} outlines the main results developed in this paper: atomic decomposition, immersion of tent spaces in the classical Gaussian function spaces in $\mathbb{R}^n$, the theory of duality in this context, and the independence of tent spaces from the aperture of the cones and the position of the cutoff function. In Section \ref{conclusiones} we give a summary of the achievements obtained and specify some open problems to develop. Section 5 proves the claims made in Section 2 and Sections 6, 7, and 8 deal with the technical proofs of the core of the paper, i.e. the atomic decomposition of the space $T^{1,q}_{\alpha, \beta}(\gamma)$, $1\le q\le \infty,$ an immersion theorem, which is considered as an application of this topic, the dual space characterization of each Gaussian tent space, and finally the independence of the spaces \( T^{p,q}_{\alpha, \beta}(\gamma) \) of $\alpha$ and $\beta$.

Throughout the paper, we adopt the following notation. For $y\in\mathbb{R}^n$, set
\[
\gamma(y)=e^{-|y|^2},\qquad d\gamma(y)=\gamma(y)\,dy.
\]

\noindent For $c_B\in\mathbb{R}^n$ and $r_B>0$, let
$ 
B=B(c_B,r_B)=\{y\in\mathbb{R}^n:\ |y-c_B|<r_B\},$  $
\overline{B}=\{y\in\mathbb{R}^n:\ |y-c_B|\le r_B\}$ and
$B^c=\mathbb{R}^n\setminus B,
$ 
and write $B(a,r)$ for the open ball of center $a$ and radius $r$. Here, $|\cdot|$ denotes the Euclidean norm on $\mathbb{R}^n$ ($n\ge1$). If $E\subset\mathbb{R}^n$ is Lebesgue measurable, we set
\[
\gamma(E):=\int_{\mathbb{R}^n}\mathcal{X}_E(y)\,d\gamma(y)=\int_{\mathbb{R}^n}\mathcal{X}_E(y)\,e^{-|y|^2}\,dy.
\]
If $E$ is a subset of a topological space $X$, we denote by $\mathring{E}$ its interior, i.e., the largest open set contained in $E$, and with $\overline{E}$ we denote the closure of $E,$ i.e., the smallest closed set that contains $E$.

\noindent We denote by $\mathscr{C}(\mathbb{R}^{n+1}_+)$ the space of continuous functions on $\mathbb{R}^{n+1}_+$ and $\mathscr{C}_c(\mathbb{R}^{n+1}_+)$ for the subset of $\mathscr{C}(\mathbb{R}^{n+1}_+)$ with compact support.
 We write $\mathscr{C}_0(\mathbb{R}^{n+1}_+)$ for the space of continuous functions on $\mathbb{R}^{n+1}_+$ that vanish at infinity. 
\noindent We also use $\mathscr{C}^1(\mathbb{R}^n)$ for continuously differentiable functions on $\mathbb{R}^n$ and the analogous
$\mathscr{C}_c^1(\mathbb{R}^{n}).$

\noindent For a measurable function $f:\mathbb{R}^n\to\mathbb{R}$, we denote by $\widehat{f}$ its Fourier transform.
And for $1\le p<\infty$, $L^p(\gamma)$ is the space of (equivalence classes of) measurable real-valued functions $f$ in $\mathbb{R}^n$ such that $|f(y)|^p$ is integrable with respect to the Gaussian measure. 
 Since $\gamma$ is absolutely continuous with respect to the Lebesgue measure and viceversa, $L^\infty(\gamma)$ coincides with the classical $L^\infty(\mathbb{R}^n,dy)$.

\noindent For $1\le p<\infty$, $L^p(\mathbb{R}^{n+1}_+,d\gamma\,\tfrac{dt}{t})$ denotes the space of measurable functions $f$ defined in $\mathbb{R}^{n+1}_+,$ such that
\[
\iint_{\mathbb{R}^{n+1}_+} |f(y,t)|^p\, d\gamma(y)\,\frac{dt}{t}<\infty.
\]
We also set $L^p_c(\mathbb{R}^{n+1}_+,d\gamma\,\tfrac{dt}{t})$ for those measurable $f$ with compact support $K\subset\mathbb{R}^{n+1}_+$ 
when we replace $\mathbb{R}^{n+1}_+$ by $K$ in the above integral.
For $p=\infty$, we will use $L^\infty(\mathbb{R}^{n+1}_+,d\gamma\,dt)$; since $d\gamma(y)\,dt$ and $dy\,dt$ are equivalent measures in $\mathbb{R}^{n+1}_+$, this space is in agreement with $L^\infty(\mathbb{R}^{n+1}_+,dy\,dt)$.

For $a,b\in\mathbb{R}$, we set $a\wedge b=\min(a,b)$ and $a\vee b=\max(a,b)$, with $a\wedge(+\infty)=a$. For real-valued functions $f,g$ with common domain we define pointwise
\[
(f\wedge g)(x)=f(x)\wedge g(x),\qquad (f\vee g)(x)=f(x)\vee g(x).
\]
Finally, $X\lesssim Y$ means that $X\le C\,Y$ for some absolute constant $C>0$; when the constant depends on parameters $\alpha$, we write $X\lesssim_\alpha Y$.

\section{Preliminaries}\label{prelim}
In this preliminary section, we first present the key notions and results that lead to an atomic decomposition of certain Gaussian tent spaces, and then introduce Gaussian Carleson measures, which are essential for understanding the duality theory associated with one of these spaces.

\subsection{Area function and Tent spaces}
In this section, we gather the main definitions and preliminary notions needed to later develop an atomic decomposition of tent spaces, as well as a complete characterization of dual spaces in the Gaussian setting.
First, in Definition~\ref{cutting}, we define admissible balls in this context. In Definition~\ref{edison}, we introduce Gaussian tent spaces, which will be proven to be Banach spaces in the Lemma~\ref{Bannach}.
To define Gaussian tent spaces, we need the concept of an area function in this context (Definition~\ref{area}) that, in turn, involves the concept of a Gaussian cone whose definition can be found in
Definition~\ref{tiendas}. We also define Gaussian tents that are the natural sets in which the supports of the Gaussian atoms 
are contained.
  At the same time, we compare the Gaussian tents with the classical ones (Lemma~\ref{GaussTent}).
  
  Finally, in Definition \ref{ATOMOS}, we define atoms that are the building blocks of an atomic decomposition.

\begin{defi}\label{cutting} Let $m$ be a cutoff function defined as follows
\[m(x)=\left\{\begin{array}{ccl}
     1 & \textrm{ if} & |x|\leq 1; \\
     \frac{1}{|x|} & \textrm{ if} & |x|\geq 1.
\end{array}\right.\] Therefore, $m(x)=1\wedge \frac{1}{|x|}.$

For $\beta>0$, let us define $m_\beta (y)=\beta m(y),$ and consider the family $\mathscr{B}_\beta$ consisting of \textit{admissible balls at level $\beta$}, that is, balls of the form $B=B(c_B,r_B)$ with $r_B\leq  m_\beta (c_B)$. For simplicity, we may write $\mathscr{B}=\mathscr{B}_1$.
\end{defi}

{
In the proof of some theorems, we are going to use
the following maximal functions:
The centered Gaussian Hardy-Littlewood maximal function on $\mathscr{B}_{\lambda},$ $\lambda>0$ defined as
$$\mathcal{M}^c_{\lambda} f (x):=\sup_{0<r<m_{\lambda}(x)}\frac{1}{\gamma(B(x,r))}\int_{B(x,r)} |f(y)|\, d\gamma (y),$$ 
and the  non-centered Gaussian Hardy-Littlewood maximal function in $\mathscr{B}_\lambda$ defined as
 $$\mathcal{M}_\lambda f(x)=\underset{B\in \mathscr{B}_\lambda: x\in B}{\sup}\frac{1}{\gamma(B)}\int_B |f(y)|\, d\gamma(y).$$

 \begin{remark}\label{wilfredo}
Clearly, the centered maximal function is bounded by the non-centered one. The non-centered Gaussian Hardy-Littlewood maximal function (and therefore the centered one) is of weak type $(1,1)$ with respect to the Gaussian measure and bounded on $L^p (\gamma),$ for $ 1<p\le \infty$ (see \cite{ur}).
 \end{remark}

\begin{defi}\label{tiendas} (See Figure \ref{uno}) Given $\alpha>0,$ $\beta>0,$ and $x\in \mathbb{R}^n$ we define two types of  \emph{Gaussian cones} with vertex at $x$, and aperture $\alpha,$
\begin{equation}\label{GaussCone1}
\Gamma_\gamma^{\alpha,\beta}(x)=\left\{(y,t)\in \mathbb{R}^{n+1}_+: |y-x|<\alpha t\wedge m_\beta(y)\right\},
\end{equation}
when $\alpha=\beta=1$, we simply write $\Gamma_\gamma(x):=\Gamma_\gamma^{1,1}(x);$ and
\begin{equation}\label{GaussCone2}
\widetilde{\Gamma}_{\gamma}^{\alpha, \beta}(x):=\left\{(y,t)\in \mathbb{R}^{n+1}_+: |y-x|<\alpha t\wedge m_\beta (x)\right\},
\end{equation}
and when $\alpha=\beta=1$, we simply write $\widetilde{\Gamma}_{\gamma}(x):=\widetilde{\Gamma}_{\gamma}^ {1,1}(x).$

\begin{figure}
\begin{center}
\includegraphics[scale=1]{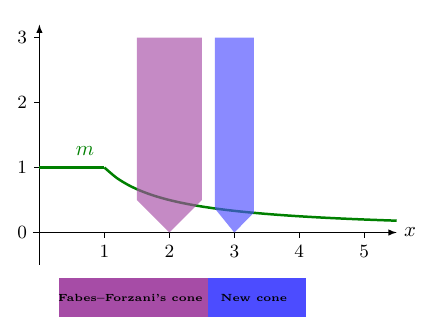}
\caption{Fabes-Forzani's cone $\widetilde{\Gamma}_\gamma(2)$ vs new cone $\Gamma_\gamma(3)$}
\label{uno}
\end{center}
\end{figure}

For a subset $F\subseteq \mathbb{R}^n$ we define
\begin{equation}
R^{\alpha,\beta}_\gamma(F)=\bigcup_{x\in F} \Gamma_\gamma^{\alpha,\beta}(x)=\bigcup_{x\in F}\left\{(y,t)\in \mathbb{R}^{n+1}_+: |y-x|<\alpha t\wedge m_\beta(y)\right\},
\end{equation}
and for $O$ another subset of $\mathbb{R}^n,$ we define the \emph{Gaussian tent over} $O$ \emph{with aperture} $\alpha$ as
\begin{eqnarray}
 \nonumber  T^{\alpha, \beta}_\gamma(O)&=&\left(R^{\alpha, \beta}_\gamma (O^c)\right)^c\\
   \nonumber    &=&\bigcap_{x\in O^c}\left\{(y,t)\in \mathbb{R}^{n+1}_+: |y-x|\geq\alpha t\wedge m_\beta (y)\right\}\\
    &=&\left\{(y,t)\in \mathbb{R}^{n+1}_+: \hbox{dist}(y,O^c)\geq\alpha t\wedge m_\beta (y)\right\}.
\end{eqnarray}
Whenever $\alpha=\beta=1$ we omit the superscripts written in the Gaussian tent. 

\noindent In most cases, $F$ will be closed and $O$ will be open subsets of $\mathbb{R}^n.$ See Figures \ref{dos} and \ref{tres} with $O=B(c_B, r_B),$ $c_B=2, m(c_B)=.5, r_B=.49,$ for figure \ref{dos} and $r_B=.6$ for figure \ref{tres}.
\end{defi}

Now, let us turn our attention to the comparison between a Gaussian tent and that one studied in the classical harmonic context. The following lemma shows that it is the same tent as the classical one as long as we take a $\beta$-admissible ball as the basis upon which the tent is set up.
\begin{lemma}\label{GaussTent}
Given $B\in \mathscr{B}_\beta,$ let us define $q_B$ the point in $\bar{B}$ such that $|q_B|=\inf_{y\in B}|y|,$ and assume that $\beta\ge 1$ and $|q_B|\ge \sqrt{\beta}.$ Then $$T_\gamma^{\alpha, \beta}(B)\subset D_{\alpha, \beta}\cup (\{c_B\}\times (0, \infty)),$$ where $D_{\alpha,\beta}=\{(y,t)\in \mathbb{R}^{n+1}_+: \alpha t<m_\beta(y)\}.$ Moreover,
\begin{equation*}
T_\gamma^{\alpha, \beta}(B)=
\left\{\begin{array}{ccl}
T^\alpha(B) & \text{if} & r_B<m_\beta(c_B)\\
T^\alpha (B)\cup (\{c_B\}\times (0, \infty)) & \text{if} & r_B=m_\beta(c_B)
\end{array}\right. .
\end{equation*}
where  $T^\alpha(B)=\{(y,t)\in \mathbb{R}^{n+1}_+:\text{dist}(y,B^c)\ge \alpha t\}$ is the classical tent over $B.$
\end{lemma}

\noindent The proof can be found in Section~\ref{LASTENT}.

 \begin{figure}
\begin{center}
\includegraphics[scale=1.7]{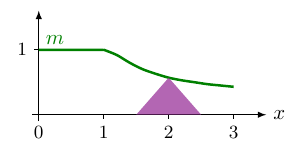}
\caption{Gaussian tent for $r_B < m(c_B)$ }
\label{dos}
\end{center}
\end{figure}

\begin{figure}
\begin{center}
\includegraphics[scale=1.7]{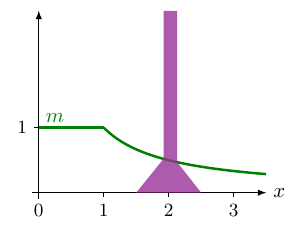}
\caption{Gaussian tent for $r_B \ge m(c_B)$ }
\label{tres}
\end{center}
\end{figure}

\begin{defi}\label{area}
Given a measurable function $f:\mathbb{R}^{n+1}_+\to \mathbb{R},$ for $\alpha, \beta >0,$ and $1\le q<\infty,$ let us consider the following area function with aperture $\alpha>0$, 
\begin{equation}\label{area4}
S_{q,\alpha, \beta}f(x)=\left(\iint_{\Gamma_\gamma^{\alpha, \beta}(x)}\frac{|f(y,t)|^q}{\gamma(B(y,\alpha t\wedge m_\beta(y)))}\,d\gamma(y)\,\frac{dt}{t}\right)^{1/q}.
\end{equation}
For $q=\infty,$ we consider
\begin{equation}
S_{\infty, \alpha, \beta}f(x)=\sup\{|f(y,t)|: (y,t)\in \Gamma_\gamma^{\alpha, \beta}(x)\},
\end{equation}
here we take $f\in \mathscr{C}(\mathbb{R}^{n+1}_+).$
Another function which we will be dealing with is
\begin{equation}
C_{q,\alpha, \beta}f(x)= \sup_{B: \; x\in B(c_B, \alpha r_B\wedge m_\beta(c_B)) }\left\{
\left(\frac{1}{\gamma(B)}\iint_{T_{\gamma}^{\alpha,\beta}(B)}|f(y.t)|^q\, d\gamma(y)\frac{dt}{t}\right)^{1/q}  \right\}.
\end{equation}
\end{defi}
We leave out the subindices $\alpha$ and $\beta$ whenever they are equal to one.
 
\begin{remark}
In \cite{FSU}, we studied the $L^p(\gamma)$-boundedness, $1<p<\infty,$ of the area function
$$\tilde{S}_2\, g(x)=\left(\iint_{\widetilde{\Gamma}_\gamma(x)}|t\nabla P^tg(y)|^2  (t^{-n}\vee |x|^n\vee 1)\,dy\, \frac{dt}{t}\right)^{1/2},$$ where $P^t$ represents the Gaussian Poisson semigroup and $\nabla$ stands for the gradient operator.  
It is immediate to prove that, pointwise, the area function $\tilde{S}_2 \, g$ is between $S_{2,1/2,1}f$ and $S_{2,2,1}f$ for $f=|t \nabla P^t g|.$
\end{remark}
 
In the proof of an atomic decomposition, we need some kind of continuity of the area function. For that, the semi-continuity of such a function is proved in the next lemma.

\begin{lemma}\label{continuity}
For every measurable function $f$ defined in $\mathbb{R}^{n+1}_+,$ the functions $S_{q,\alpha, \beta, } f,$ for $1\le q<\infty$ and $C_{q, \alpha, \beta}f,$ for $1<q<\infty$ are lower semi-continuous. If $f$ is further continuous, then $S_{\infty, \alpha, \beta}f$ turns out to be also lower semi-continuous.
\end{lemma}
\noindent  The proof can be found in Section \ref{falta}.

Now, let us turn to the definition of the Gaussian tent space.

\begin{defi}\label{edison}  Given $ \alpha, \beta >0,$ $1 \leq p <\infty$ and $1 \leq q< \infty$, the Gaussian tent space $T^{p,q}_{\alpha, \beta} (\gamma)$ is the set of measurable functions $f:{\mathbb R}^{n+1}_+\to \mathbb{R}$ such that $S_{q,\alpha, \beta}f\in L^p(\gamma)$, i.e.,
$$T^{p,q}_{\alpha, \beta} (\gamma)=\left\{f:{\mathbb R}^{n+1}_+\to \mathbb{R}:\  S_{q,\alpha, \beta}f\in L^p(\gamma)\right\},$$
 with norm
\[\|f\|_{T_{\alpha,\beta}^{p,q}(\gamma)}:=\|S_{q,\alpha, \beta}f\|_{L^p(\gamma)}.\]
For $q=\infty,$ then \[T^{p,\infty}_{\alpha,\beta}(\gamma)=\{f\in \mathscr{C}(\mathbb{R}_+^{n+1}):\  S_{\infty, \alpha,\beta}f\in L^p(\gamma) 
\},\]
with norm
\[\| f\|_{T^{p, \infty}_{\alpha, \beta}(\gamma)}:=\|S_{\infty, \alpha,\beta}\, f\|_{L^p(\gamma)}.\]
Finally,
for $p=\infty$ and $1< q<\infty$ we have
{
\[ T_{\alpha, \beta}^{\infty, q}(\gamma)=\{f:\mathbb{R}^{n+1}_+\to \mathbb{R}:\  C_{q, \alpha, \beta}f\in L^\infty(\gamma)\},\]}
{with norm $$\|f\|_{T^{\infty, q}_{\alpha, \beta}(\gamma)}:=\|C_{q,\alpha,\beta}\, f\|_{L^\infty(\gamma)}.$$}
Whenever $\alpha=\beta=1$ we omit the subscripts written on the Gaussian tent space.
\end{defi}

\begin{remark}It can be seen immediately that the space $T^{p,p}_{\alpha, \beta}(\gamma)=L^p(\mathbb{R}^{n+1}_+, d\gamma(y)\frac{dt}{t}),$ for $1\le p<\infty.$ 
\end{remark}

We next establish the completeness of the Gaussian tent spaces with respect to their natural norms.

\begin{lemma}\label{Bannach}
For every $1\le p<\infty,$ $1\le q\le \infty,$ and $\alpha,\beta>0,$ the tent space $T_{\alpha, \beta}^{p,q}(\gamma)$ turns out to be a Banach space with the norm $\|\cdot\|_{T_{\alpha, \beta}^{p,q}(\gamma)}.$
For $p=\infty$ and $1< q <\infty$ the tent space $T^{\infty, q}_{\alpha, \beta}(\gamma)$ is also a Banach space.
\end{lemma}
\noindent Its proof can be found in Section~\ref{proofbanach}.

We now define the Gaussian atoms that are the building blocks of the atomic decomposition in the Gaussian context.

\begin{defi}\label{ATOMOS}
For $1\le q\le \infty$, let $\alpha>0,$ $\beta>0$ and $\delta>0$ be given. A measurable function $a:\mathbb{R}^{n+1}_+\to \mathbb{R}$ is said to be a $T^{1,q}_{\alpha, \beta}(\gamma)$ $\delta$-atom if there exists a ball $B\in \mathscr{B}_{\delta}$ such that, for $1\le q<\infty,$
\begin{enumerate}
    \item \label{atomsupport}$\hbox{supp} (a)\subset T_\gamma^{\alpha, \beta} (B)$;
    \item \label{normatom}$\|a\|_{L^q\left(\mathbb{R}^{n+1}_+,d\gamma \, \frac{dt}{t}\right)}\leq \gamma(B)^{-(1-1/q)}$.  
\end{enumerate}
For $q=\infty$, we maintain the item~\ref{atomsupport}, add the condition to the atom that it must be continuous, and item~\ref{normatom} becomes $\|a\|_{L^\infty(\mathbb{R}^{n+1}_+,d\gamma dt)}\le \gamma(B)^{-1}.$
\end{defi}

\begin{remark}
Relationship between the atoms obtained in \cite{MNP} and the ones presented here:
According to what we have defined we observe that the atoms in \cite{MNP} are the restriction of our atoms to the region $D=\{(y,t)\in \mathbb{R}^{n+1}_+: 0<t<m(y)\},$ which we may call the ``local" region in $\mathbb{R}_+^{n+1}.$ Thus the tent in \cite{MNP}, we will call it $T_\gamma^D(B)$ for $B$ an admissible ball, will be the intersection of our tent $T_\gamma(B)$ with the region $D.$
\end{remark}

As is common in this context of tent spaces, we are going to prove that these atoms are in the closed unit ball of the space $T^{1,q}_{\alpha, \beta}(\gamma).$

\begin{lemma}\label{NOSE}If $a$ is a $T^{1,q}_{\alpha, \beta}(\gamma)$ $\delta$-atom, then $a\in T^{1,q}_{\alpha, \beta} (\gamma)$ and
\[\|a\|_{T^{1,q}_{\alpha, \beta}(\gamma)}=\|S_{q,\alpha, \beta}\, a\|_{L^1(d\gamma)}\leq 1.\]
\end{lemma}

The proof can be found in Section \ref{atomitos}.

\subsection{Carleson Measure}

In this section, we also introduce the notion of a Gaussian Carleson measure, which plays a key role in the duality theory of tent spaces. This measure will be characterized in terms of its behavior in tents over admissible balls. Finally, we introduce the space $\mathscr{M}_C$ of Carleson measures.

\begin{defi} A signed measure $\mu$ defined on a Borel $\sigma$-algebra in $\mathbb{R}^{n+1}_+$ is said to be a Gaussian Carleson measure if  there exist a $\delta>\beta>0,$ and a constant $C$ such that for any ball $B\in \mathscr{B}_\delta,$
\begin{equation}
|\mu|(T_\gamma^{\alpha, \beta}(B)) \leq C  \gamma(B),
\end{equation}
being $|\mu|=\mu_++\mu_-$ the total variation of $\mu$ (see \cite[Chapters 2 \& 6]{ru}).

Equivalently,
\begin{equation}
\sup_{B\in \mathscr{B}_\delta}\frac{1}{ \gamma(B)} \iint_{T^{\alpha, \beta}_\gamma(B)} d|\mu|(x,t) \leq C.
\end{equation}
\end{defi}
The set $\mathscr{M}_C$ that collects all Carleson measures is a Banach space with the norm
\[\|\mu\|_C=\sup_{B\in \mathscr{B}_\delta}\frac{1}{ \gamma(B)} \iint_{T^{\alpha,\beta}_\gamma(B)} d|\mu|(x,t).\]

\section{Main results}\label{main}
 
In this section, we present the three main contributions of the paper. First, we prove the atomic decomposition of Gaussian tent spaces \( T^{1,q}_{\alpha, \beta}(\gamma) \) and explore their connections to classical Gaussian function spaces using a suitable operator. Then, we establish the duality theory for Gaussian tent spaces. This framework is analogous to the classical duality theory in harmonic analysis, where dual spaces are identified through natural pairings. In particular, at the endpoint \( p = 1, q = \infty \), the notion of Gaussian Carleson measures emerges as the appropriate dual objects in this setting. The last contribution has to do with the proof of the independence of tent spaces from the Gaussian cone aperture and from the cutoff function position, that is, the tent spaces do not depend neither on \( \alpha \) nor on \( \beta \), that is,
\[
T^{p,q}_{\alpha, \beta}(\gamma) = T^{p,q}_{1, 1}(\gamma), \quad \text{for all } \alpha \text{ and } \beta.
\] and it will be interpreted as $T^{p,q}(\gamma).$

\subsection{Atomic decomposition}\label{atomic}

In this section, we present the atomic decomposition of the space $T^{1,q}_{\alpha, \beta}(\gamma)$ for $1\le q \le \infty,$ and establish a connection between the Gaussian tent spaces $T^{1,2}(\gamma),$ $T^{p,2}(\gamma)$ for $1<p<\infty,$ $T^{\infty,2}(\gamma)$ and the corresponding $\mathbb{R}^n$-spaces $H^1(\gamma),$ $L^p(\gamma)$ for $1<p<\infty,$ $BMO(\gamma)$ through a bounded linear operator similar to the one defined in \cite{cms}.

\begin{teo}[Atomic decomposition]\label{teo: descomposicion atomica} For every $\alpha>0$, $\beta>0$, $ 1\le q \le +\infty $ and every $f\in T_{\alpha, \beta}^{1,q}(\gamma)$, there exist a sequence $\{\lambda_m\}_{m\in \mathbb{N}}\in \ell^1$ and a sequence of $T^{1,q}_{\alpha, \beta}(\gamma)$ $\delta$-atoms $\{a_m\}_{m\in \mathbb{N}}$ for some $\delta>0$ such that
\begin{enumerate}
    \item $f=\sum_{m=1}^\infty \lambda_m a_m$;
    \item $\sum_{m=1}^\infty |\lambda_m|\leq C\|f\|_{T_{\alpha,\beta}^{1,q}(\gamma)}$ for some positive constant $C$ independent of $f$.
\end{enumerate}
\end{teo}

  The proof of the Theorem can be found in Section \ref{elteorema}.

\begin{remark}
The converse of Theorem \ref{teo: descomposicion atomica} is also true using the fact that the atoms are in the unit ball of $T^{1,q}_{\alpha,\beta}(\gamma)$ according to Lemma \ref{NOSE}, the series associated with the sequence $\{\lambda_m\}$ is absolutely convergent, and $T^{1,q}_{\alpha,\beta}(\gamma)$ is a Banach space by Lemma \ref{Bannach}.
\end{remark}

We close this section with Theorem \ref{coneccion1}, which describes how Gaussian tent spaces are embedded into the classical function spaces associated with the Gaussian measure. The existence of such immersions allows us to interpret several properties of tent spaces within the more familiar framework of spaces like \( H^1(\gamma) \), \( L^p(\gamma) \), and \( BMO(\gamma) \).

\begin{teo}\label{coneccion1}
Immersions into classical Gaussian function spaces.
For every \( 1 \le p \le \infty \), the Gaussian tent space \( T^{p,2}(\gamma) \) embeds continuously into the corresponding classical Gaussian function space on \( \mathbb{R}^n \):
\begin{enumerate}
\item \( T^{1,2}(\gamma) \hookrightarrow H^1(\gamma); \)
\item \( T^{p,2}(\gamma) \hookrightarrow L^p(\gamma) \), for \( 1 < p < \infty \);
\item \( T^{\infty,2}(\gamma) \hookrightarrow BMO(\gamma). \)
\end{enumerate}
\end{teo}

The proof of Theorem~\ref{coneccion1} is given in Section~\ref{edison1}.

\begin{remark}
The definitions of \( H^1(\gamma) \) and \( BMO(\gamma) \) can be found in~\cite{MM}. 
What we really prove in 1. is that the immersion function takes the atoms from
Theorem \ref{teo: descomposicion atomica} to the $H^1(\gamma)$ atoms defined in \cite{MM}.
\end{remark}

\subsection{Duality}\label{carleson}

In this section, we characterize the dual of the endpoint space $T^{1,\infty}_{\alpha, \beta} (\gamma)$ in terms of Gaussian Carleson measures. Then, we generalize this result by identifying the dual of $T^{p,q}_{\alpha, \beta} (\gamma)$, for $1 \leq p < \infty$ and $1 < q < \infty$, with the space $T^{p',q'}_{\alpha, \beta} (\gamma)$ through a natural pairing. These duality results {also provide another way of proving that such Gaussian tent spaces are necessarily Banach spaces.}

\begin{teo}\label{car1}
The dual of $T^{1,\infty}_{\alpha, \beta} (\gamma)$ is the space $\mathscr{M}_C$ of Carleson measures, via the pairing
\begin{equation}
(\mu, f)\to \iint_{\mathbb{R}^{n+1}_+} f(x,t)\, d\mu(x,t),
\end{equation}
for every $(\mu, f)\in \mathscr{M}_C\times T^{1,\infty}_{\alpha, \beta}(\gamma).$
\end{teo}

The proof can be found in Section \ref{car1proof}.
\begin{teo}\label{car2}
For every $1\le p<\infty$ and $1<q<\infty,$ the dual space of $T^{p,q}_{\alpha, \beta}(\gamma)$ can be identified to $T^{p',q'}_{\alpha, \beta}(\gamma)$ via the pairing
\begin{equation}
(g,f)\to \iint_{\mathbb{R}^{n+1}_+} g(y,t)\, f(y,t)\, d\gamma (y)\frac{dt}{t},
\end{equation}
for $(g,f)\in T^{p',q'}_{\alpha,\beta}(\gamma)\times T^{p,q}_{\alpha, \beta}(\gamma).$
\end{teo}
 
The proof can be found in Section \ref{car2proof}.

\subsection{Independence of  $\alpha$ and $\beta$}\label{new}

 The following theorem states that tent spaces are actually independent of the parameters $\alpha$ and $\beta$. 
  
 In fact, we could have started just with $\alpha=\beta=1$ in their definition.  

\begin{teo}\label{independencia}
For either $1\le p,q<\infty,$ or $1< p<\infty, q=\infty,$ or $p=\infty, 1<q<\infty$ and for every  measurable function $f$ on $\mathbb{R}^{n+1}_+$ in the first and third cases and every continuous function $f$ in the second case the following equivalence of norms  $$\|f\|_{T_{\alpha, \beta}^{p,q}(\gamma)}\sim_{n, \alpha, \beta, \delta, \lambda}\, \|f\|_{T^{p,q}_{\delta, \lambda}(\gamma)}$$ holds for any $\alpha, \beta, \delta, \lambda>0.$
\end{teo}

The proof can be found in Section \ref{car3proof}
\section{Conclusions and open problems}\label{conclusiones}
In the classical setting of $(\mathbb{R}^{n+1}_+, dx\,\tfrac{dt}{t})$, Coifman-Meyer-Stein \cite{cms} introduced the tent spaces $T^{p,q}$ and developed a unified framework that includes atomic decompositions, duality, and interpolation. 
In this work, we treat the Gaussian counterpart: we define tent spaces on $(\mathbb{R}^{n+1}_+, d\gamma \,\tfrac{dt}{t})$ and establish an atomic decomposition and the full duality theory.

In the Gaussian setting, the dual of $T^{1,q}(\gamma)$ is described via the functional $C$ through $T^{\infty,q'}(\gamma)$, while for $1<p<\infty$ the dual of $T^{p,q}(\gamma)$ is described via the conical square function $S$ through $T^{p',q'}(\gamma)$. 
In analogy to \cite{cms}, it is natural to investigate the precise relationship between $S$ and $C$. 
We expect that $\|S(f)\|_{L^{p}(\gamma)}$ and $\|C(f)\|_{L^{p}(\gamma)}$ are equivalent for $2<p<\infty$, which would allow a uniform characterization of $T^{p,q}(\gamma)$ by the functional $C$ in the whole range $2<p<\infty$. 
Another direction, also suggested by \cite{cms}, is to develop the complex interpolation method for the Gaussian tent spaces $T^{p,q}(\gamma)$. On the other hand, another open problem is characterizing $H^1(\gamma)$ through the area function, either as the one in \cite{FSU} or as the one in \cite{MNP2}, belonging to space $L^1(\gamma).$

\section{Proofs of Section \ref{prelim}}\label{demostraciones}

\subsection{Proof of Lemma\ref{GaussTent}}\label{LASTENT}

Let $h(s)=s+\frac{\beta}{s}$ with $s>0$ then $h'(s)>0$ if and only  $s>\sqrt{\beta}.$ Thus $h$ is strictly increasing for $s\ge \sqrt{\beta}.$

Given $B\in \mathscr{B}_\beta,$ let us assume that there exists $(y,t)\in T^{\alpha, \beta}_\gamma(B)$ and $(y,t)\notin D_{\alpha, \beta},$  i.e., $\alpha t\ge m_\beta (y)$ and $\text{dist}(y, B^c)\ge \alpha t\wedge m_\beta (y); $
i.e. $r_B-|y-c_B|\ge m_\beta (y).$

Since $|q_B|\ge 1,$ then for all $y\in B, |y|\ge 1$ and $m_\beta (y)=\frac{\beta}{|y|},$ in particular, 
$m_\beta (c_B)=\frac{\beta}{|c_B|}.$ Taking into account all the above and that $r_B \le m_\beta (c_B) $ we have
$$|y|-|c_B|\le |y-c_B|\le r_B-\frac{\beta}{|y|}\le \frac{\beta}{|c_B|}-\frac{\beta}{|y|}.$$ That is $h(|y|)\le h(|c_B|).$ Since $h$ is strictly increasing for $ s\ge \sqrt{\beta}$ and $|y|$ and $|c_B| \ge |q_B| \ge \sqrt{\beta}$, we get that $|y|\le |c_B|.$  

In the same way, we now have $|c_B|-|y|\le r_B-\frac{\beta}{|y|}\le \frac{\beta}{|c_B|}-\frac{\beta}{|y|}$. Using that the function $k(s)=s-\frac{\beta}{s}$ is positive and strictly increasing for $s>\sqrt{\beta}$ we obtain $|c_B|\le |y|.$ Thus $|y|=|c_B|$ and so $m_\beta (y)=m_\beta (c_B),$ and $|y-c_B|\le r_B - m_\beta (y)\le 0.$ Therefore, $y=c_B.$
Thus, the first part of the lemma is proven.

For the second part, it is immediate to see that $T^\alpha (B)$ is always a subset of $T_\gamma^{\alpha, \beta}(B).$
Now, from the proof of the first part,
for all $(y,t)\in T^{\alpha, \beta}_\gamma (B)$ with $y\ne c_B$ we have $(y,t)\in D_{\alpha, \beta},$ and then $(y,t)\in T^\alpha (B).$

For $y=c_B,$ we have $\text{dist}(c_B,B^c)=r_B\le m_\beta(c_B),$ and $(c_B,t)\in T_\gamma^{\alpha, \beta}(B),$ implies $m_\beta (c_B)\ge r_B =\text{dist}(c_B,B^c)\ge \alpha t\wedge m_\beta(c_B).$ Now, if $r_B<m_\beta(c_B),$ then $r_B\ge \alpha t,$ i.e., $(c_B,t)\in T^\alpha(B).$ If $r_B=m_\beta(c_B),$ then $\text{dist}(c_B,B^c)=r_B=m_\beta(c_B)\ge \alpha t\wedge m_\beta (c_B)$ for all $t>0,$ thus $(c_B,t)\in T^\alpha(B)\cup (\{c_B\}\times (0,\infty)).$  This ends the proof of Lemma \ref{GaussTent}.

{{ 
\subsection{Proof of Lemma \ref{continuity}}\label{falta}

Let $h>0$ be given. For $\alpha,$ $\beta>0$ we define the cones
\begin{equation*}
\Gamma_{\gamma, h}^{\alpha,\beta}(x)=\{(y,t)\in\mathbb{R}^{n+1}_+: |y-x|<\alpha(t-h)\wedge (m_\beta(y)-\alpha h)\}.
\end{equation*}
This is a family of Gaussian cones so that $\bigcup_{h>0}\Gamma_{\gamma, h}^{\alpha,\beta}(x)=\Gamma_{\gamma}^{\alpha, \beta}(x),$ and for $1\le q<\infty,$ we also have
\begin{equation}\label{limite}
S_{q,\alpha,\beta}f(x)=\limsup_{h\to 0^+}\left(\iint_{\Gamma_{\gamma, h}^{\alpha, \beta}(x)}\frac{|f(y,t)|^q}{\gamma(B(y,\alpha t\wedge m_\beta (y)))}\, d\gamma(y)\, \frac{dt}{t}\right)^{1/q}.
\end{equation}
For $\lambda>0$, let us call $O_1=\{x\in \mathbb{R}^n: S_{q,\alpha,\beta}f(x)>\lambda\}$ for $1\le q\le \infty,$ and $O_2=\{x\in \mathbb{R}^n: C_{q,\alpha, \beta} f(x)>\lambda\}$ for $1<q<\infty.$

In order to prove that these functions are all lower semicontinuous it suffices to prove that $O_1$ and $O_2$ are open sets.

For $1\le q<\infty,$ and for every $x\in O_1$, taking into account (\ref{limite}), there exists $h_0>0$ so that $$\left(\iint_{\Gamma_{\gamma, h_0}^{\alpha, \beta}(x)}\frac{|f(y,t)|^q}{\gamma(B(y,\alpha t\wedge m_\beta (y)))}\, d\gamma(y)\, \frac{dt}{t}\right)^{1/q}>\lambda.$$ We also have that for any $  z\in B(x,\alpha h_0),$ $\Gamma_{\gamma, h_0}^{\alpha, \beta}(x)\subset \Gamma_\gamma^{\alpha, \beta}(z).$ Indeed, for $(y,t)\in \Gamma_{\gamma, h_0}^{\alpha, \beta}(x),$ $|y-z|<|y-x|+\alpha h_0<\alpha (t-h_0)\wedge(m_\beta(y)-\alpha h_0)+\alpha h_0\le \alpha t\wedge m_\beta(y).$ Hence, $(y,t)\in \Gamma_\gamma^{\alpha,\beta}(z).$ Thus, $B(x,\alpha h_0)\subset O_1.$ Therefore, $O_1$ is open.

For $q=\infty,$ let us recall that we are asking $f\in \mathscr{C}(\mathbb{R}^{n+1}_+).$ For every $x\in O_1$
there exists $(y,t)\in \Gamma_\gamma^{\alpha, \beta}(x)$ so that $|f(y,t)|>\lambda.$ Observe that for every $z\in B(y,\alpha t\wedge m_\beta(y)),$ we have $S_{\infty, \alpha, \beta}f(z)\ge |f(y,t)|>\lambda.$ So $x\in B(y,\alpha t\wedge m_\beta (y))\subset O_1.$ Thus $O_1$ is open in this case as well.

Now, for every $x\in O_2$ there exists an open ball $B$ such that $x\in B(c_B, \alpha r_B\wedge m_\beta (c_B))$ and $$\left(\frac{1}{\gamma(B)}\iint_{T^{\alpha, \beta}_\gamma(B)}|f(y,t)|^q\, d\gamma(y)\, \frac{dt}{t}\right)^{1/q}>\lambda.$$ Thus, $x\in B(c_B, \alpha r_B\wedge m_\beta (c_B))\subset O_2.$ Hence $O_2$ is also an open set.

\subsection{Proof of Lemma \ref{Bannach}}\label{proofbanach}

For the proof of Lemma \ref{Bannach} and other results we need the following lemma whose proof can be found in [Lemma 2.3 \cite{MNP}].

  \begin{lemma}\label{comparacion}
Let $b>0$ be given. If $|x-y|<b\, m(y)$ then 
$$m(y)<(b+1)m(x), \quad  \text{and} \quad m(x)<(b+1)m(y).$$
\end{lemma}

We also need to introduce subspaces that are dense in Gaussian tent spaces. We denote by $L_c^q(\mathbb{R}^{n+1}_+, d\gamma dt/t)$ the space of $L^q$-integrable functions with compact support.
\begin{lemma}\label{density}
The space $L^q_c(\mathbb{R}^{n+1}_+,d\gamma(y)dt/t)$ is a dense subspace of $T^{p,q}_{\alpha, \beta}(\gamma),$ for $1\le p<\infty$ and $1\le q<\infty$. On the other hand, $\mathscr{C}_c(\mathbb{R}^{n+1}_+)$ is a dense subspace of $T^{p,\infty}_{\alpha, \beta}(\gamma).$
\end{lemma}

\noindent {\bf Case
$1\le q<\infty$, $1\le p <\infty$} 

We first prove that $L_c^{q}(\mathbb{R}^{n+1}_+,d\gamma\, dt/t) \subset T^{p,q}_{\alpha,\beta}(\gamma)$. In fact, letting $1\le q<\infty$ and $g \in 
 L_c^{q}(\mathbb{R}^{n+1}_+,d\gamma\, dt/t)$, then there exists a compact set $K$ such that
$\text{supp}(g)\subseteq K$ and $g \in L^q(K, d\gamma dt/t)$. Without loss of generality, we may assume that $K= \bar{B}\left(0, r\right) \times[a, b]$ for some open ball $B\left(0, r\right) \subset \mathbb{R}^n$ and $a, b \in(0, \infty)$.
Note that $S_{q,\alpha, \beta}(g)$ is supported on $\bar{B}\left(0, r+\alpha b\right) \subset \mathbb{R}^n$, which is a bounded subset of $\mathbb{R}^n$. In fact, for $|x|>r+\alpha b$ if $(y,t)\in K,$ then $|y-x|\ge |x|-|y|>\alpha b\ge \alpha t\ge \alpha t\wedge m_\beta(y),$ that is, $(y,t)\notin \Gamma_\gamma^{\alpha, \beta}(x).$
Now, let us prove that $g\in T_{\alpha, \beta}^{p,q}(\gamma).$
\begin{align*}
\|g\|_{T_{\alpha,\beta}^{p,q}(\gamma)}^p &  =\|S_{q,\alpha, \beta}\, g\|^p_{L^p(\gamma)} \\
& \leq \int_{B\left(0, r+\alpha b\right) }\left(\iint_{\Gamma_\gamma^{\alpha, \beta}(x)\cap K}\frac{|g(y,t)|^q}{\gamma(B(y,\alpha t\wedge m_\beta(y)))}\,d\gamma(y)\,\frac{dt}{t}\right)^{\frac{p}{q}} d \gamma(x) \\
& \leq C \int_{B\left(0, r+\alpha b\right)}\left(\frac{e^{|x|^2}}{\left[\alpha a\wedge \frac{m_\beta(x)}{\beta+1}\right]^{n}}\right)^{p/q} \left(\iint_{K} |g(y, t)|^q d\gamma(y)\, \frac{dt}{t}\right)^{\frac{p}{q}} dx\\
& \leq  C_{p,q,\alpha,\beta, K}\|g\|_{L^q(K, d\gamma\, dt/t)}^p .
\end{align*}

Hence $L_c^{q}(\mathbb{R}^{n+1}_+,d\gamma\, dt/t)$ is contained in $T^{p,q}_{\alpha,\beta}(\gamma).$

Now, if we take $f\in T^{p,q}_{\alpha,\beta}(\gamma),$ we will prove that for
any compact set $K$ we have
\begin{equation}\label{mm}
 \left\|f \mathcal{X}_K \right\|_{L^q(\mathbb{R}^{n+1}_+,d\gamma dt/t)} \leq   C_{n, \alpha, \beta, K}\ \|f\|_{T^{p, q}_{\alpha, \beta}(\gamma)}.
\end{equation}
  In fact, due to the compactness of $\bar{B}(0,r),$ we split $\bar{B}(0,r)$ into a finite open cover of balls of radius $\frac{1}{2} (\alpha a \wedge \frac{m_\beta(x)}{\beta+1})$, i.e., pick $x_1, \ldots, x_N$ such that $\bar{B}\left(0, r\right) \subset \bigcup_{j=1}^N B\left(x_j, \frac{1}{2} (\alpha a \wedge \frac{m_\beta(x_j)}{\beta+1})\right)$. And with the choice of these balls we have
$$B\left(x_j, \frac{1}{2}(\alpha a \wedge \frac{m_\beta(x_j)}{\beta+1})\right) \subset \left\{ x \in \mathbb{R}^n: B\left(x_j, \frac{1}{2}(\alpha a\wedge \frac{m_\beta(x_j)}{\beta+1}) \right) \times[a, b] \subset \Gamma_\gamma^{\alpha, \beta}(x) \right\},$$
 which ensures an average estimate on each of them as
\begin{align*}
 & \left(\iint_{B\left(x_j, \frac{1}{2}(\alpha a \wedge \frac{m_\beta(x_j)}{\beta+1})\right) \times[a, b]}\frac{|f(y,t)|^q}{\gamma(B(y,\alpha t\wedge m_\beta(y)))}\,d\gamma(y)\,\frac{dt}{t}\right)^{\frac{1}{q}}  \\
&  \leq \left(  \; \avgint_{B\left(x_j, \frac{1}{2}(\alpha a \wedge \frac{m_\beta(x_j)}{\beta+1}) \right)}\right. \left.\left(\iint_{\Gamma_\gamma^{\alpha, \beta}(x)}\frac{|f(y,t)|^q}{\gamma(B(y,\alpha t\wedge m_\beta(y)))}\,d\gamma(y)\,\frac{dt}{t}\right)^{\frac{p}{q}} d\gamma(x)\right)^{\frac{1}{p}} \\
 & \lesssim C_{n, p, \alpha, \beta, a, r}\left(\int_{B\left(x_j, \frac{1}{2} (\alpha a \wedge \frac{m_\beta(x_j)}{\beta+1})\right)}\left(S_{q,\alpha, \beta}(f)(x)\right)^p   d\gamma(x)\right)^{\frac{1}{p}} \\
&\leq C_{n, p, \alpha, \beta, K}  \|f\|_{T_{\alpha,\beta}^{p,q}(\gamma)},
\end{align*}
where $\avgint_A f(x)\, d\gamma(x)=\frac{1}{\gamma(A)} \int_A f(x)\, d\gamma(x).$
Adding these pieces and setting $C_n =\pi^{n/2}$, we get
\begin{align*}
&\left\|f {\mathcal X}_K \right\|_{L^q(\mathbb{R}^{n+1}_+, \gamma dt/t)} \leq C_n^{\frac{1}{q}}\left(\iint_K \frac{|f(y,t)|^q}{\gamma(B(y,\alpha t\wedge m_\beta(y)))}\,d\gamma(y)\,\frac{dt}{t}\right)^{\frac{1}{q}} \\
& \leq C_n \sum_{j=1}^N \left(\iint_{B\left(x_j, \frac{1}{2} (\alpha a\wedge \frac{m_\beta (x_j)}{\beta+1})\right) \times[a, b]}\frac{|f(y,t)|^q}{\gamma(B(y,\alpha t\wedge m_\beta(y)))}\,d\gamma(y)\,\frac{dt}{t}\right)^{\frac{1}{q}} \\
& \leq C_{n, \alpha, \beta, K}\ \|f\|_{T^{p, q}_{\alpha, \beta}(\gamma)}.
\end{align*}
That proves \eqref{mm}.

Now, we fix a sequence of compact subsets $K_m$ of $\mathbb{R}^{n+1}_+$ that form an exhaustion of $\mathbb{R}_{+}^{n+1}$, that is, $K_m \subset$ $K_{m+1}$ for all $m \in \mathbb{N}$ and $\bigcup_{m \in \mathbb{N}} K_m=\mathbb{R}_{+}^{n+1}$. Then, let us
take $f\in T^{p,q}_{\alpha,\beta}(\gamma),$ and consider a sequence $\{g_m\}_{m\in {\mathbb N}}\subset L^{q}_c(\mathbb{R}^{n+1}_+, d\gamma(y)\frac{dt}{t})$ defined as $g_m(y,t)=f(y,t)\mathcal{X}_{K_m}(y,t).$ 

Next, we have to prove that $g_m\to f$ as $m\to \infty$ in $T^{p,q}_{\alpha, \beta}(\gamma),$ that is, $\|S_{q,\alpha, \beta}(f-g_m)\|_{L^p(\gamma)}\to 0$ as $m\to \infty.$  

We start by noticing that the sequence $\{|f(y,t)-g_m(y,t)|\}_{m\in \mathbb{N}}$ is a nonnegative decreasing sequence bounded above by $|f(y,t)|$ that converges to $0$ as $m\to \infty.$ Taking into account that $S_{q,\alpha, \beta} f\in L^p(\gamma),$ we find that $S_{q,\alpha, \beta} f(x)$ is finite for almost every $x\in \mathbb{R}^n.$ So, by Beppo Levy's monotone convergence theorem $\lim_{m\to \infty}S_{q,\alpha,\beta}(f-g_m)(x)=0$
for almost every $x\in \mathbb{R}^n.$ And then again by applying the same theorem to the non-negative monotone decreasing sequence $\left\{S_{q,\alpha,\beta}(f-g_m)(x)\right\}$
bounded above by
$S_{q,\alpha, \beta}f(x)$ a.e. $x\in \mathbb{R}^n$, $\lim_{m\to \infty}\|S_{q,\alpha, \beta}(f-g_m)\|_{L^p(\gamma)}=0$ holds.

\noindent {\bf Case
$  q=\infty$, $1\le p <\infty$} 

Let $g$ be a continuous function in $\mathbb{R}^{n+1}_+$ such that $\text{supp} (g)\subseteq K$. 

Observe that $S_{\infty, \alpha, \beta}\, g(x)\le \|g\|_\infty$ for every $x\in \mathbb{R}^n,$  and taking into account that $\gamma$ is a finite measure, we conclude that $\|g\|_{T^{p,\infty}_{\alpha, \beta}(\gamma)}\le C_n \|g\|_\infty.$
Thus $\mathscr{C}_c(\mathbb{R}^{n+1}_+)\subset T^{p,\infty}_{\alpha, \beta}(\gamma).$

Given $f\in T^{p,\infty}_{\alpha, \beta}(\gamma),$ to get a continuous function with compact support that is close to $f$ in $T^{p,\infty}_{\alpha, \beta}(\gamma)$, we must produce a smooth truncation for $f$. For the given compact set $K$ we take $g_K=f\eta_K$ with $\eta_K$ being a mollification of the characteristic function $\mathcal{X}_K.$ In this case $\eta_K\in C_c(\mathbb{R}^{n+1}_+)$ with its support contained a small compact neighborhood of $K.$  
Proceeding as before, 
we obtain
$$
\left\|f \eta_K\right\|_{\infty}  \le C_{N,n,\alpha, \beta, K}\|f\|_{T^{p, \infty}_{\alpha, \beta}(\gamma)}.
$$
Having set a sequence of compact sets $\{K_m\}_{m\in {\mathbb N}} $ that exhausts $\mathbb{R}^{n+1}_+,$ as we have done above, for $f\in T^{p,\infty}_{\alpha, \beta}(\gamma)$ we define the sequence $\{f\eta_{K_m}\}_{m\in {\mathbb N}}$ in $\mathscr{C}_{c}(\mathbb{R}^{n+1}_+)$ approaching $f$ in $T^{p,\infty}_{\alpha, \beta}(\gamma)$. The proof of this claim follows the same steps as those done above for $1\le q<\infty.$

There is no density result for the spaces $T^{\infty,q}_{\alpha, \beta}(\gamma),$ $1<q<\infty.$

 \

\noindent {\bf Proof of Lemma \ref{Bannach}.}
The proof is based on the localization of a function: Lemma \ref{density}. Let $\left\{f_k\right\}_{k \in \mathbb{N}}$ be a Cauchy sequence in $T_{\alpha, \beta}^{p,q}(\gamma)$.

For the case $1 \le p, q<\infty$, by Lemma \ref{density}, the sequence $\left\{f_k \mathcal{X}_{K}\right\}_{k \in \mathbb{N}}$ turns out to be a Cauchy sequence in $L^q(K, d\gamma dt/t)$ for every compact set $K \subset \mathbb{R}_{+}^{n+1}$, which therefore converges to some $f_K \in L^q(K, d\gamma dt/t)$. Pick $\left\{K_m\right\}_{m \in \mathbb{N}}$ a sequence of compact sets as an exhaustion of $\mathbb{R}_{+}^{n+1}$. Define $f=\lim _{m \rightarrow \infty} f_{K_m}$ pointwise, which is well defined due to the following identification: For any $i, j \in \mathbb{N},$ with $i>j$, $f_{K_i} \mathcal{X}_{K_j}=f_{K_j}$. 

Let $\epsilon>0$ be given, taking into account that $\{f_k\}_{k\in {\mathbb N}}$ is a Cauchy sequence in $T^{p,q}_{\alpha, \beta}(\gamma)$ together with Lemma \ref{density}, there exist $N, m \in \mathbb{N}$ large enough so that $\left\|f_k-f_N\right\|_{T_{\alpha, \beta}^{p,q}(\gamma)}<\epsilon$ for every $k \geq N,$ $\left\|f-f \mathcal{X}_{K_m}\right\|_{T_{\alpha,\beta}^{p, q}(\gamma)}<\epsilon$, and $\left\|f_N-f_N \mathcal{X}_{K_m}\right\|_{T_{\alpha, \beta}^{p,q}(\gamma)}<\epsilon$. Thus, for every $k \geq N$,
\begin{align*}
\left\|f-f_k\right\|_{T_{\alpha, \beta}^{p,q}(\gamma)} & \leq\left\|f-f \mathcal{X}_{K_m}\right\|_{T_{\alpha, \beta}^{p,q}(\gamma)}+\left\|\left(f-f_k\right)
\mathcal{X}_{K_m}\right\|_{T_{\alpha, \beta}^{p,q}(\gamma)}\\
& +\left\|\left(f_k-f_N\right) \mathcal{X}_{K_m}\right\|_{T_{\alpha, \beta}^{p,q}(\gamma)}  +\left\|f_N \mathcal{X}_{K_m}-f_N\right\|_{T_{\alpha, \beta}^{p,q}(\gamma)}\\
&+ \left\|f_N-f_k\right\|_{T_{\alpha, \beta}^{p,q}(\gamma)} \\
& <4 \epsilon+\left\|\left(f-f_k\right) \mathcal{X}_{K_m}\right\|_{T_{\alpha, \beta}^{p,q}(\gamma)}\\ & \lesssim_m  4\epsilon+ \|f_{K_m}-f_k\mathcal{X}_{K_m}\|_{L^q(K_m, d\gamma dt/t)}.
\end{align*}
We conclude by taking $\limsup$ on both sides as $k\to \infty$. It remains to prove that $f \in T_{\alpha, \beta}^{p,q}(\gamma)$. In fact,
take $N$ so that $\|f-f_N\|_{T^{p,q}_{\alpha,\beta}(\gamma)}<1,$ then $\|f\|_{T^{p,q}_{\alpha,\beta}(\gamma)}<1+\|f_N\|_{T^{p,q}_{\alpha,\beta}(\gamma)}<\infty.$

The proof for case $1\le p <\infty, q=\infty$ is the same. Extra continuity follows from the Banach property of $\mathscr{C}(K)$ on each compact set $K$.

 For the case $p=\infty$, $1< q <\infty.$ Let $\{f_k\}_{k\in {\mathbb N}}$ be a Cauchy sequence in $T^{\infty, q}_{\alpha, \beta}(\gamma).$ Then there exists $N$ a subset of $\mathbb{R}^n$ of measure zero such that $\hbox{ for any } x\in \mathbb{R}^n\setminus N,$
$C_{q,\alpha, \beta}(f_k)(x)\le \|f_k\|_{T^{\infty, q}_{\alpha, \beta}(\gamma)}$ for all $k\ge 1;$ being $$C_{q,\alpha, \beta}(f_k)(x)=\sup_{B: x\in B_{\alpha, \beta}}\frac{1}{\gamma(B)^{1/q}}\|f_k\|_{L^q(T^{\alpha, \beta}_\gamma(B),d\gamma dt/t)},$$ and $B_{\alpha, \beta}=B(c_B, \alpha r_B\wedge m_\beta (c_B)).$

Let $\{B_j\}_{j\in {\mathbb Z}^+}$ be a sequence of balls with center at $0$ and radius $j.$ Then $\{T^{\alpha,\beta}_\gamma(B_j)\}_j$ is an increasing sequence of tents such that $\bigcup_j T^{\alpha,\beta}_\gamma(B_j)=\mathbb{R}^{n+1}_+.$ Notice that $\{f_k\}_{k\in {\mathbb N}}$ is a Cauchy sequence in $L^q\left(T^{\alpha, \beta}_\gamma(B_j), d\gamma dt/t\right)$, hence it converges to some $f_{B_j} \in L^q\left(T^{\alpha, \beta}_\gamma(B_j), d\gamma dt/t\right)$.  

Now, as we have discussed, we construct the limit function $f$ as a measurable function such that $f\mathcal{X}_{T^{\alpha,\beta}_\gamma (B_j)}=f_{B_j},$ since for $i>j,$ $f_{B_i}\mathcal{X}_{T^{\alpha,\beta}_\gamma(B_j)}=f_{B_j}.$

Let us prove that $f\in T_{\alpha, \beta}^{\infty, q}(\gamma).$ Since $\{f_k\}_{k\in {\mathbb N}}$ is a Cauchy sequence in $T^{\infty, q}_{\alpha, \beta}(\gamma),$ then it is bounded, i.e., $\sup_{k\in {\mathbb N}}\|f_k\|_{T^{\infty, q}_{\alpha, \beta}(\gamma)}<\infty.$
 Let us prove that $C_{q,\alpha, \beta}f(x)\le \sup_{k\in {\mathbb N}}\|f_k\|_{T^{\infty, q}_{\alpha, \beta}(\gamma)}$ for almost every $x\in \mathbb{R}^n.$ 
 
 In fact, for every open ball $B$, there exists $j\in {\mathbb N}$ such that $B\subset B_j\subseteq B_i,$ for all $i\ge j.$ Then
\begin{align*}
\left(\frac{1}{\gamma(B)} \right. & \left.\iint_{T^{\alpha, \beta}_\gamma(B)}|f(y, t)|^q d\gamma(y) \frac{dt}{t}\right)^{\frac{1}{q}} = \left(\frac{1}{\gamma(B)} \iint_{T^{\alpha, \beta}_\gamma(B)}|f_{B_j}(y, t)|^q d\gamma(y) \frac{dt}{t}\right)^{\frac{1}{q}}\\ &\leq \sup _{k \geq 1}\left(\frac{1}{\gamma(B)} \iint_{T^{\alpha, \beta}_\gamma(B)}\left|f_k(y, t)\right|^q d\gamma(y) \frac{d t}{t}\right)^{\frac{1}{q}}\\ & \leq \sup _{k \geq 1}\ C_{q,\alpha, \beta}(f_k) (x) \ \text{for\ } x\in B_{\alpha, \beta}.
\end{align*}
Thus, \begin{align*}
C_{q,\alpha, \beta}f(x)&\le \sup_{k\in {\mathbb N}}\ C_{q,\alpha, \beta}(f_k)(x) \\ &\le \sup_{k\in {\mathbb N}}\ \left\|f_k\right\|_{T_{\alpha, \beta}^{\infty,q}(\gamma)}\ \text{for a.e.}\ x\in \mathbb{R}^n.
\end{align*}
so $f \in T_{\alpha, \beta}^{\infty,q}(\gamma)$.

Now, let us prove $\|f_k-f\|_{T^{\infty, q}_{\alpha, \beta}(\gamma)}\to 0$ as $k\to \infty.$ 
Note that
\begin{align*}
C_{q,\alpha, \beta}(f-f_k)(x)& = \sup _{B:\ x\in B_{\alpha, \beta}} \gamma(B)^{-\frac{1}{q}}\left\|f-f_k\right\|_{L^q\left(T^{\alpha, \beta}_\gamma(B) , d \gamma \frac{dt}{t}\right)} \\
& \leq \sup _{B:\ x\in B_{\alpha, \beta}} \gamma(B)^{-\frac{1}{q}}\left[\limsup _{\substack{m \rightarrow \infty \\
m \geq k}}\left(\left\|f-f_m\right\|_{L^q\left(T^{\alpha, \beta}_\gamma(B) , d \gamma \frac{dt}{t}\right)} \right. \right.\\
& \;\;\; \left. \left. +\left\|f_m-f_k\right\|_{L^q\left(T^{\alpha, \beta}_\gamma(B) , d \gamma \frac{dt}{t}\right)}\right)\right] \\
& \leq \sup _{B:\ x\in B_{\alpha, \beta}} \limsup _{\substack{m \rightarrow \infty \\
m \geq k}} \gamma(B)^{-\frac{1}{q}}\left\|f-f_m\right\|_{L^q\left(T^{\alpha, \beta}_\gamma(B) , d \gamma \frac{dt}{t}\right)}\\
& \;\;\;  +\sup _{m \geq k}\left\|f_m-f_k\right\|_{T_{\alpha, \beta}^{\infty,q}(\gamma)}\\
& \leq \sup _{m \geq k}\left\|f_m-f_k\right\|_{T_{\alpha, \beta}^{\infty,q}(\gamma)}, \text{for a.e.}\ x\in \mathbb{R}^n.
\end{align*}

We conclude the proof since $\{f_k\}$ is a Cauchy sequence.

\subsection{Proof of Lemma \ref{NOSE}}\label{atomitos}

If $a$ is a $T^{1,q}_{\alpha, \beta}(\gamma)$ $\delta$-atom, it is supported on $T_\gamma^{\alpha, \beta}(B)$ for some $B\in \mathscr{B}_\delta$. Notice that for any $(y,t)\in T_\gamma^{\alpha, \beta}(B)$ and $x\in B(y, \alpha t\wedge m_\beta(y))$, we have $x\in B$. In fact, since $(y,t)\in T_\gamma^{\alpha, \beta} (B)$, $\alpha  t\wedge m_\beta(y)\leq \hbox{dist}(y,B^c)$, and, on the other hand, we know that $|x-y|<\alpha t\wedge m_\beta(y)$. Thus, $|y-x|< d(y,B^c)$ means that $x\notin B^c$, i.e., $x\in B$ as claimed.

Now, let us estimate $\|S_{q,\alpha, \beta}a\|_{L^1(\gamma)}$. For $1\le q<\infty,$ by using the above fact, H\"older's inequality and Fubini's Theorem, we get
\begin{align*}
   & \|S_{q, \alpha, \beta}a\|_{L^1(\gamma)}=\int_{\mathbb{R}^n} \left(\iint_{\Gamma_\gamma^{\alpha, \beta}(x)} \frac{|a(y,t)|^q}{\gamma(B(y,\alpha t\wedge m_\beta(y)))} d\gamma(y)\frac{dt}{t}\right)^{1/q} d\gamma(x)\\
    &=\int_{\mathbb{R}^n} \left(\iint_{T_\gamma^{\alpha, \beta} (B)} \mathcal{X}_{B(y,\alpha t\wedge m_\beta(y))}(x) \frac{|a(y,t)|^q}{\gamma(B(y,\alpha t\wedge m_\beta(y)))} d\gamma(y)\frac{dt}{t}\right)^{1/q} d\gamma(x)\\
    &=\int_{\mathbb{R}^n} \left(\iint_{\mathbb{R}^{n+1}_+} \mathcal{X}_{B(y,\alpha t\wedge m_\beta(y))}(x) \frac{|a(y,t)|^q}{\gamma(B(y,\alpha t\wedge m_\beta(y)))} d\gamma(y)\frac{dt}{t}\right)^{1/q} \mathcal{X}_B(x) d\gamma(x)\\
    &\leq \left(\int_{\mathbb{R}^n} \iint_{\mathbb{R}^{n+1}_+} \mathcal{X}_{B(y,\alpha t\wedge m_\beta(y))}(x) \frac{|a(y,t)|^q}{\gamma(B(y,\alpha t\wedge m_\beta(y)))} d\gamma(y)\frac{dt}{t} d\gamma(x)\right)^{1/q} \gamma(B)^{1/q'}\\
    &=\left(\iint_{\mathbb{R}^{n+1}_+} \frac{|a(y,t)|^q}{\gamma(B(y,\alpha t\wedge m_\beta(y)))} \int_{\mathbb{R}^n}  \mathcal{X}_{B(y,\alpha t\wedge m_\beta(y))}(x) d\gamma(x) d\gamma(y)\frac{dt}{t} \right)^{1/q} \gamma(B)^{1/q'}\\
    &=\left(\iint_{\mathbb{R}^{n+1}_+}|a(y,t)|^q d\gamma(y)\frac{dt}{t} \right)^{1/q} \gamma(B)^{1/q'}\\
    &=\|a\|_{L^q(\mathbb{R}^{n+1}, d\gamma \frac{dt}{t})}\gamma(B)^{1/q'}\leq 1.
\end{align*}
For $q=\infty,$ \begin{align*}\|S_{\infty, \alpha, \beta}f\|_{L^1(\gamma)}&=\int_{B}\sup\{|a(y,t)|:(y,t)\in \Gamma^{\alpha,\beta}_\gamma(x)\}\, d\gamma(x)\\ &\le \|a\|_{L^\infty(\mathbb{R}^{n+1}_+, d\gamma dt)}\gamma(B)\le 1.\end{align*}\qedhere

 \section{Proofs of Section \ref{atomic}}

\subsection{Preliminary results for the proof of Theorem \ref{teo: descomposicion atomica}}

To prove Theorem \ref{teo: descomposicion atomica}, we need some definitions and some lemmas.

We define the $\eta$-density points of a measurable set $A$ associated to $\lambda$-admissible balls.

Given $A$ a measurable subset of ${\mathbb R}^n$, $\eta>0,$ and $\lambda >0$ we write
\begin{equation}\label{unonuevo}A^{[\eta]}_{\lambda}=\left\{x\in \mathbb R^n: \frac{\gamma(A\cap B)}{\gamma(B)}\geq \eta,  \hbox{ for any } B=B(x,r)\in \mathscr B_\lambda\right\}.\end{equation}

We call $A_\lambda^{[\eta]}$ the set of \textit{$\eta$-density points of  $A$} associated to \textit{$\lambda$-admissible balls}. Note that $A_\lambda^{[\eta]}$ is a closed subset of $\mathbb R^n$ that happens to be contained in $\overline{A}$.

\begin{lemma}\label{cerrado} For any real number $\eta>0$ and $\lambda>0$, $A^{[\eta]}_\lambda$ is a closed subset of $\mathbb R^n$ and $A^{[\eta]}_\lambda \subseteq \overline{A}$.
\end{lemma}
\begin{proof}
Let $\{x_j\}_{j\in{\mathbb N}}$ be a sequence in $A_\lambda^{[\eta]}$ such that $x_j\to x$ as $j\to \infty$. Fix $0<r<\lambda m(x)$. For the cutoff function $m$ being continuous on $\mathbb{R}^n$, then there exists $N\in \mathbb N$ such that for all $j\geq N$, $0<r<\lambda m(x_j)\leq \lambda m(x)$. Since $x_j\in A_\lambda^{[\eta]}$, we have
\[\frac{\gamma(A\cap B(x_j,r))}{\gamma(B(x_j,r))}\geq \eta,\]
for each $j\geq N$. On the other hand, as $x_j\to x$, we also get
\begin{equation}\label{densidad}
    \frac{\gamma(A\cap B(x,r))}{\gamma(B(x,r))}\geq \eta.
\end{equation}

Now, if $r=\lambda m(x)$, we have $0<r-\varepsilon<r=\lambda m(x)$ for any $0<\varepsilon<r$. Taking into account the case $r<\lambda m(x)$,
\[\frac{\gamma(A\cap B(x,r-\varepsilon))}{\gamma(B(x,r-\varepsilon))}\geq \eta,\]
for every $0<\varepsilon<r$. Due to the continuity of the measure $\gamma$, \eqref{densidad} also holds for $r=\lambda m(x)$.  Thus, $x\in A_\lambda^{[\eta]}$ and this proves that $A_\lambda^{[\eta]}$ is closed.

To see that $A_\lambda^{[\eta]}\subseteq \overline{A}$, we take $x\in A_\lambda^{[\eta]}$ and a neighborhood $V$ of $x$. Then, there exists a ball $B=B(x,r)\subset V$ with $r\leq \lambda m(x)$ such that $\gamma(B\cap A)\geq \eta\, \gamma(B)$. Thus, $A\cap B\neq \emptyset$. Therefore, $V\cap A\neq \emptyset$ for any neighborhood $V$ of $x$. Thus, $x\in \overline{A}$.\\
\end{proof}

\begin{lemma}\label{lema: desigualdad integral, relacion entre eta y beta}
For every $\eta\in (0,1)$ there exists $\bar{\eta}\in (0,1)$ such that for every measurable subset $A$ of $\mathbb R^n$ and every non-negative measurable function $H$ on $\mathbb R^{n+1}_+$, we have
\begin{align*}
    \iint_{R^{(1-\eta)\alpha, (1-\eta)\beta}_\gamma\left(A_{\beta(1+\beta)}^{[\bar{\eta]}}\right)} & H(y,t) d\gamma(y) \frac{dt}{t} \\
    &\lesssim \int_A \left(\iint_{\Gamma_\gamma^{\alpha, \beta}(x)} \frac{H(y,t)}{\gamma(B(y,\alpha t\wedge m_\beta(y))} d\gamma(y) \frac{dt}{t}\right) d\gamma(x).
\end{align*}
\end{lemma}

\begin{proof}
By applying Fubini's Theorem, we get
\begin{align*}
    \int_A &\left(\iint_{\Gamma_\gamma^{\alpha, \beta}(x)} \frac{H(y,t)}{\gamma(B(y,\alpha t\wedge m_\beta(y)))} d\gamma(y) \frac{dt}{t}\right) d\gamma(x) \\
    &= \iint_{\mathbb R^{n+1}_+} \left(\int_A \mathcal{X}_{B(y,\alpha t\wedge m_\beta (y))}(x) d\gamma(x)\right)  \frac{H(y,t)}{\gamma(B(y,\alpha t\wedge m_\beta(y)))} d\gamma(y) \frac{dt}{t}\\
    &= \iint_{\mathbb R^{n+1}_+} \frac{\gamma(A\cap B(y,\alpha t\wedge m_\beta (y)))}{\gamma(B(y,\alpha t\wedge m_\beta (y)))} H(y,t) d\gamma(y) \frac{dt}{t}.
\end{align*}
We wish to find a suitable subset of ${\mathbb R}^{n+1}_+$ so that the quotient $\frac{\gamma(A\cap B(y,\alpha t\wedge m_\beta(y)))}{\gamma(B(y,\alpha t\wedge m_\beta(y)))}$ be bounded below by a constant on such a subset and  in this way to obtain the desired inequality.

First, notice that for every $(y,t)\in R_{(1-\eta)\alpha, (1-\eta)\beta}\left(A^{[\bar{\eta}]}_{\beta (1+\beta)}\right)$, there exists $x\in A^{[\bar{\eta}]}_{\beta(1+\beta)}$ such that ${|x-y|<(1-\eta)(\alpha t\wedge m_\beta(y))}$. Thus, from Lemma~\ref{comparacion} with $b=\beta$, $ m(y)\leq (1+\beta)\, m(x)$. 

From the above fact, $\alpha t\wedge m_\beta(y) \leq \beta\,  m(y)\leq \beta (1+\beta)\,  m(x)$, so we have $B(x,\alpha t\wedge m_\beta(y))\in \mathscr{B}_{\beta (1+\beta)}$, and since $x\in A^{[\bar{\eta}]}_{\beta (1+\beta)}$, then 
\[\frac{\gamma(A\cap B(x,\alpha t\wedge m_\beta(y)))}{\gamma(B(x,\alpha t\wedge m_\beta(y)))}\geq \bar{\eta}.\]

Therefore,
\begin{align*}
  &\gamma(A\cap B(y,\alpha t\wedge m_\beta(y)))\geq \gamma(A\cap B(x,\alpha t\wedge m_\beta(y)))\\ &\ \ \ \ -\gamma(B(x,\alpha t\wedge m_\beta(y))\cap B^c(y,\alpha t\wedge m_\beta(y)))\\
  &\geq \bar{\eta} \gamma(B(x,\alpha t\wedge m_\beta(y)))-\gamma(B(x,\alpha t\wedge m_\beta(y))\cap B^c(y,\alpha t\wedge m_\beta(y))).
\end{align*}

From the relationship between $x$ and $y$, we get $B(x,\eta (\alpha t\wedge m_\beta(y)))\subseteq B(y,\alpha t\wedge m_\beta(y))$. On the other hand, from the doubling property of the Gaussian measure on admissible balls from $\mathscr{B}_{\beta (1+\beta)}$, we know that there exists a constant $C=C(1/\eta, \beta)>1$ such that
\[\gamma(B(x,\alpha t\wedge m_\beta(y)))\leq C \gamma(B(x,\eta (\alpha  t\wedge m_\beta(y)))).\]
Hence,
\begin{align*}
\gamma(B(x,\alpha t\wedge m_\beta(y))\cap B(y,\alpha t\wedge m_\beta(y)))&\geq \gamma(B(x,\eta (\alpha t\wedge m_\beta(y))))\\ &\geq \frac{1}{C}\gamma(B(x, \alpha t\wedge m_\beta(y))).\end{align*}
Consequently, 
\begin{align}\label{eq: gamma de F intersec la bola}
  \nonumber \gamma(A\cap B(y,\alpha t\wedge m_\beta(y)))&\geq \bar{\eta} \gamma(B(x,\alpha t\wedge m_\beta(y))) \nonumber \\ &\ \ \ \ -\left(1-\frac{1}{C}\right)\gamma(B(x,\alpha t\wedge m_\beta(y)))\nonumber\\
  &=\left(\bar{\eta}-1+\frac{1}{C}\right)\gamma(B(x,\alpha t\wedge m_\beta(y))).
\end{align}

Again, since $|y-x|<\beta\, m(y)$ and $ m(y)\leq (1+\beta) m(x)$,
\begin{align*}
    \left||y|^2-|x|^2\right|&=\left||y|-|x|\right|\left(|x|+|y|\right)\leq \beta\, m(y)(|x|+|y|) \\
    & \leq \beta \, (1+\beta) m(x)|x|+ \beta \, m(y)|y|\\
    &\leq \beta\, (2+\beta).
\end{align*}
Then, there exists a positive constant $K=e^{-3\beta(2+\beta)}$ such that $\gamma(B(x,\alpha t\wedge m_\beta(y)))\geq K \gamma(B(y,\alpha t\wedge m_\beta(y)))$. Using this inequality in \eqref{eq: gamma de F intersec la bola}, we get 
\[\gamma(A\cap B(y,\alpha t\wedge m_\beta(y)))\geq \left(\bar{\eta}-1+\frac{1}{C}\right)K\, \gamma(B(y,\alpha t\wedge m_\beta(y))). \] Let us call $\lambda=\left(\bar{\eta}-1+\frac{1}{C}\right)K,$ then 
by choosing $\bar{\eta}$ such that $1-\frac{1}{C}<\bar{\eta} <1$, we get the desired estimate
\[\frac{\gamma(A\cap B(y,\alpha t\wedge m_\beta(y)))}{\gamma(B(y,\alpha t\wedge m_\beta(y)))}\geq \lambda,\]
for every $(y,t)\in R_{(1-\eta)\alpha, (1-\eta)\beta}\left(A^{[\bar{\eta}]}_{\beta(1+\beta)}\right)$. Therefore, 
\begin{align*} &\iint_{R_{(1-\eta)\alpha, (1-\eta)\beta}\left(A^{[\bar{\eta}]}_{\beta(1+\beta)}\right)} H(y,t) d\gamma(y) \frac{dt}{t}    \\
& \;\;\; \leq \frac{1}{\lambda} \int_A \left(\iint_{\Gamma_\gamma^{\alpha, \beta}(x)} \frac{H(y,t)}{\gamma(B(y,\alpha t\wedge m_\beta(y))} d\gamma(y) \frac{dt}{t}\right) d\gamma(x).\qedhere
\end{align*}
\end{proof}

\subsection{Proof of Theorem \ref{teo: descomposicion atomica} }\label{elteorema}

In the proof of Theorem \ref{teo: descomposicion atomica}, we will use the following definitions:
\begin{defi}
Let $\lambda>0$ be given. A set $W\subseteq \mathbb{R}^n$ is said to be a Whitney set $\lambda$-admissible if $\hbox{ for any } x\in W,$ we have $\text{dist}(x,W^c)\le \lambda\, m(x),$ where $m(x)$ is the function in Definition \ref{cutting}
\end{defi}
\begin{defi}
Let $A$ be a subset of $\mathbb{R}^n$ and $\lambda>0$ we define the set 
$$A+\mathscr{C}_\lambda=\{c_B: B=B(c_B,r_B)\in \mathscr{B}_\lambda\ \text{and}\ B\cap A\ne \emptyset \},$$ which collects all the centers of balls $\lambda$-admissible that intersect the set $A.$ 
\end{defi}
 \begin{remark}
Observe that $A+\mathscr{C}_\lambda$ is always an open set. In fact, if $c_B\in A+\mathscr{C}_\lambda,$ then there exists $a\in B(c_B, r_B)\cap A.$ And since $a\in B(b,r_B)$ for all $b\in B(c_B, r_B-|a-c_B|),$ then $B(c_B, r_B-|a-c_B|)\subset A+\mathscr{C}_\lambda.$ Thus, the claim follows.
 \end{remark}

In the process of getting an atomic decomposition of a function, we use an important tool which is the Whitney covering lemma. This technique consists in finding a covering of every proper open subset of $\mathbb{R}^n$ by a sequence of dyadic cubes whose diameters are proportional to the distance of the cube from the complementary of that open set. The inconvenience we encounter in the Gaussian context is that we use admissible dyadic cubes to cover each open set and these cubes become very small in size when they are far away from the origin, and we lose the property of proportionality of the size of that cube compared with its distance from the complementary of the open set. To overcome this, in \cite{MNP}, they construct open sets where this covering lemma works.  

{For $1\le q<\infty,$ let} $f\in T^{1,q}_{\alpha,\beta}(\gamma)$ be given. In order to obtain an atomic decomposition of $f$ we use the construction from \cite{MNP} where they get a finite number of disjoint sets $W$ covering ${\mathbb R}^n$, and such that $W+\mathscr{C}_{2^p}$ are $\lambda$-admissible Whitney open sets with $\lambda=2^{2p+2}\sqrt{n}$ for every natural number $p\ge 2.$  

Therefore, $f$ can be expressed as the sum of $f$ restricted to each of these sets $W$ and to prove the theorem it will suffice to show that the function $g(y,t)=f(y,t)\mathcal{X}_W (y)$ can be decomposed into $T^{1,q}_{\alpha, \beta}(\gamma)$ $\delta-$atoms for some $\delta>0.$

To construct the atoms, let us start by defining for every $k\in {\mathbb Z}$,
\[O_k:=\{x\in \mathbb R^n: S_{q,\alpha, \beta}(g)(x)>2^k\},\]
and $F_k:=O_k^c$. Fix $\eta\in (0,1)$ and let $\bar{\eta}$ be the constant in $(0,1)$ found in Lemma~\ref{lema: desigualdad integral, relacion entre eta y beta}. With abuse of notation, we let $O_k^{[\bar{\eta}]}:=\left(F_{k, \beta(1+\beta)}^{[\bar{\eta}]}\right)^c$, where $F_{k, \beta(1+\beta)}^{[\bar{\eta}]}$ denotes the set of $\bar{\eta}$-density points of $F_k$ associated with $\beta(1+\beta)$-admissible balls; see (\ref{unonuevo}) for its definition. Note that $O_k\subseteq O_k^{[\bar{\eta}]}$. In fact, we know that $F_k$ is a closed subset of $\mathbb R^n$ since $O_k$ is open according to the Lemma \ref{continuity}, and $F_{k, \beta(1+\beta)}^{[\bar{\eta}]}\subseteq F_k$, according to the Lemma \ref{cerrado}, then $O_k=F_k^c\subseteq \left(F_{k, \beta (1+\beta)}^{[\bar{\eta}]}\right)^c=O_k^{[\bar{\eta}]}$.
 
{We claim that $O_k^{[\bar{\eta}]} 
\subset W+\mathscr{C}_{2^p}$ for some $p$ that depends on $\beta.$}
We first fix $x\in O_k$ and check that $x\in W+\mathscr{C}_{\beta (1+\beta)}$. In fact, for $x\in O_k$, $S_{q,\alpha, \beta} (g)(x)>2^k>0,$ and thus there exists $(y,t)\in \Gamma_\gamma^{\alpha, \beta}(x)$, such that
\[\mathcal{X}_{B(y,\alpha t\wedge m_\beta(y))}(x) g(y,t)=\mathcal{X}_{B(y,\alpha t\wedge m_\beta(y))}(x) f(y,t) \mathcal{X}_{W}(y)\neq 0.\]
Hence, $y\in W$ and $|x-y|<\alpha t\wedge m_\beta(y) \leq\beta\,  m(y)\leq \beta (1+\beta) m(x)$, according to Lemma \ref{comparacion} with $b=\beta.$ Thus $B(x,\alpha t\wedge m_\beta(y))\in \mathscr{B}_{\beta (1+\beta)}$ and $y\in B(x,\alpha t\wedge m_\beta(y))$, i.e., $W\cap B(x,\alpha t\wedge m_\beta(y))\neq \emptyset$. Thus, $x\in W+\mathscr{C}_{\beta (1+\beta)},$  and therefore $O_k\subset   W+\mathscr{C}_{\beta (1+\beta)}$.

Next, let $x\in O_k^{[\bar{\eta}]}$. Then $x$ is not a $\bar{\eta}$-density point of $F_k$, so there is a ball $B\in \mathscr{B}_{\beta(1+\beta)}$ centered at $x$ such that $\gamma(F_k\cap B)<\bar{\eta}\,  \gamma(B)$. This is only possible if $B$ intersects $O_k=F_k^c$. Since $O_k\subseteq W+\mathscr{C}_{\beta(1+\beta)}$, this means that $B$ intersects $W+\mathscr{C}_{\beta (1+\beta)}$.

Fix an arbitrary $x'\in B\cap (W+\mathscr{C}_{\beta(1+\beta)})$ and let $B'\in \mathscr{B}_{\beta(1+\beta)}$ be any admissible ball centered at $x'$ and intersecting $W$. Since $x'\in B$ and $B\in \mathscr{B}_{\beta(1+\beta)}$, it follows that $|x-x'|<r_B\le \beta (1+\beta)\, m(x).$ 

Also, since $B'$ belongs to $\mathscr{B}_{\beta (1+\beta)}$ and intersects $W$ (that is, there exists $x''\in B'\cap W$ with $\text{dist}(x',W)\leq |x'-x''|<r_{B'}\leq \beta (1+\beta)\, m(x')$), we have $\text{dist}(x',W)<\beta (1+\beta)\, m(x')$. It follows that 
\[\text{dist}(x,W)\leq |x-x'|+\text{dist}(x',W)<\beta(1+\beta)\, m(x)+\beta (1+\beta)\, m(x').\]
If we set $b=\beta(1+\beta),$ $y=x$ and $x=x'$ in Lemma \ref{comparacion}, we have $$m(x')\le (1+\beta (1+\beta))\, m(x).$$ Thus
\begin{align*}\text{dist}(x,W)&<[\beta(1+\beta)+\beta (1+\beta)(1+\beta (1+\beta))]\, m(x)\\ &<\beta (1+\beta) (2+\beta(1+\beta))\, m(x).
\end{align*}
Let us take $p$ the least non-negative integer such that $2^p\ge \beta (1+\beta)(2+\beta(1+\beta)).$ Therefore $\text{dist}(x,W)<2^p\, m(x).$
This means that there exists $z\in W$ with $|x-z|<2^p\, m(x)$, that is, $W\cap B(x,2^p m(x))\neq \emptyset$. Hence, $x\in W+\mathscr{C}_{2^p}$ for every $x\in O_k^{[\bar{\eta}]}.$ Hence $O_k^{[\bar{\eta}]}\subseteq W+\mathscr{C}_{2^p}$.

{\bf Now, we will prove that $O_{k+1}^{[\bar{\eta}]}\subset O_k^{[\bar{\eta}]}.$}
Notice that $O_k^{[\bar{\eta}]}=\{x\in \mathbb{R}^n: \mathcal{M}^c_{\beta (1+\beta)}(\mathcal{X}_{O_k})(x)>1-\bar{\eta})\}.$ Now, $\mathcal{M}^c_{\beta (1+\beta)} f(x) $ is weak-type $(1,1)$  with respect to the measure $\gamma$ (Remark \ref{wilfredo}), therefore $\gamma(O_k^{[\bar{\eta}]})\le C_{n,\bar{\eta}, \beta}\, \gamma(O_k)$ and also since $O_{k+1}\subset O_k,$ then $O_{k+1}^{[\bar{\eta}]}\subset O_k^{[\bar{\eta}]}.$

{\bf Next, we show that $\text{supp}(g)\subseteq \bigcup_{k\in {\mathbb Z}}T^{(1-\eta)\alpha, (1-\eta)\beta}_\gamma (O_k^{[\bar{\eta}]}).$}
In fact, 
from Lemma~\ref{lema: desigualdad integral, relacion entre eta y beta} applied to $H(y,t)=|g(y,t)|^q$ we get 
\begin{align}\label{soporte}
    \iint_{R_\gamma^{(1-\eta)\alpha, (1-\eta)\beta}(F_{k,\beta(1+\beta)}^{[\bar{\eta}]})} & |g(y,t)|^q   d\gamma(y) \frac{dt}{t} \nonumber\\
    &\lesssim\int_{F_{k}} \left(\iint_{\Gamma_\gamma^{\alpha, \beta}(x)} \frac{|g(y,t)|^q}{\gamma(B(y,\alpha t\wedge m_\beta(y))))} d\gamma(y) \frac{dt}{t}\right) d\gamma(x)\\
    &=\int_{F_k} (S_{q,\alpha, \beta}\, g (x))^q \, d\gamma(x)\le C_n\, 2^{q k}.\nonumber
\end{align}
Observe that $\mathbb{R}^{n+1}_+\setminus \bigcup_{k\in {\mathbb Z}}T^{(1-\eta)\alpha, (1-\eta)\beta}_\gamma (O_k^{[\bar{\eta}]})=\bigcap_{k\in {\mathbb Z}} R_\gamma^{(1-\eta)\alpha, (1-\eta)\beta} (F_{k,\beta (1+\beta)}^{[\bar{\eta}]})=:\Omega,$ then, taking into account the inequality (\ref{soporte}), $$\|\mathcal{X}_\Omega\, g\|_{L^{q}(\mathbb{R}^{n+1}_+,d\gamma\frac{dt}{t})}\le C_n\, 2^k,\ \hbox{ for any } k\in {\mathbb Z}.$$ Thus, as $k\to -\infty,$ we get $g(y,t)=0$ for almost every $(y,t)\in \Omega.$ So, the support of $g$ is contained in $$\bigcup_{k\in {\mathbb Z}}T^{(1-\eta)\alpha, (1-\eta)\beta}_\gamma (O_k^{[\bar{\eta}]})=
{\bigcup_{k\in {\mathbb Z}}\!\!\!\!\!\!\!{\hbox{\tiny D}}} \; [T^{(1-\eta)\alpha, (1-\eta)\beta}_\gamma(O_k^{[\bar{\eta}]})\setminus T_\gamma^{(1-\eta)\alpha, (1-\eta)\beta}(O_{k+1}^{[\bar{\eta}]})],$$
since $O_{k+1}^{[\bar{\eta}]}\subset O_k^{[\bar{\eta}]}.$

{\bf Proof of the atomic decomposition.}
Since $O_k^{[\bar{\eta}]}\subset  W+\mathscr{C}_{2^p}$ and $W+\mathscr{C}_{2^p}$ is an open $\lambda$- admissible Whitney set, $O_k^{[\bar{\eta}]}$ also turns out to be an open $\lambda$-admissible Whitney set.

Using the classical Whitney lemma (see \cite{Stein-Singular integrals and differentiability properties}) there exists a sequence of disjoint dyadic cubes $\{Q_j^k\}_{j\in {\mathbb N}}$ such that 

\begin{enumerate} 
\item $O_k^{[\bar{\eta}]}=\bigcup_j Q_j^k;$
\item $\text{diam} (Q_j^k)\le \text{dist}(Q_j^k, (O_k^{[\bar{\eta}]})^c)\le 4 \text{diam}(Q_j^k)$ and $d_j^k:=\text{diam}(Q_j^k)\le \lambda m (c_{Q_j^k}),$ being $c_{Q_j^k}$ the center of $Q_j^k.$
\end{enumerate}

 We have $O_k^{[\bar{\eta}]}=\bigcup_{j\in {\mathbb N}}Q_j^k$ and we define $B_j^k$ the ball with center $c_{Q_j^k}$ and radius $C>0$ times $d_j^k$ the diameter of $Q_j^k.$ The constant $C>0$ will be chosen large enough so that 
\begin{equation}\label{tiendainclusion}T^{(1-\eta)\alpha, (1-\eta)\beta}_\gamma(O_k^{[\bar{\eta}]})\subset \bigcup_{j\in {\mathbb N}}T^{\alpha, \beta}_\gamma(B_j^k)\cap (Q_j^k\times (0,\infty)).\end{equation} 

Choice of the constant $C$. Taking $(y,t)\in T^{(1-\eta)\alpha, (1-\eta)\beta}_\gamma(O_k^{[\bar{\eta}]}),$ then $(1-\eta)(\alpha t\wedge m_\beta(y)\le \text{dist}(y, (O_k^{[\bar{\eta}]})^c)$ and $y\in O_k^{[\bar{\eta}]},$ hence, there exists $j\in {\mathbb N}$ such that $y\in Q_j^k.$ So $(y,t)\in Q_j^k\times (0,\infty).$ Check that for $C>0$ large enough, $(y,t)\in T_\gamma^{\alpha, \beta}(B_j^{k}),$ as well. Given $z\notin B_j^k,$ we have $$(1-\eta)(\alpha t\wedge m_\beta (y))\le \text{dist}(y, (O_k^{[\bar{\eta}]})^c)\le \text{dist}(Q_j^k,(O_k^{[\bar{\eta}]})^c)+d_j^k\le 5 d_j^k.$$ On the other hand, $|y-z|\ge |z-c_{Q_j^k}|-|y-c_{Q_j^k}|\ge (C-1)d_j^k.$ Combining the above inequalities we obtain
$$\alpha t\wedge m_\beta (y)\le \frac{5}{(1-\eta)(C-1)}|y-z|,$$ for all $z\in (B_j^k)^c.$ Thus, by taking $C\ge 1+\frac{5}{1-\eta},$ we obtain that $(y,t)\in T_\gamma^{\alpha, \beta}(B_j^k).$ With this, the inclusion is proved.

Taking into account (\ref{tiendainclusion}), we can write $$T^{(1-\eta)\alpha, (1-\eta)\beta}_\gamma(O_k^{[\bar{\eta}]})\setminus T_\gamma^{(1-\eta)\alpha, (1-\eta)\beta}(O_{k+1}^{[\bar{\eta}]}),$$ as a disjoint union $\bigcup_{j\in {\mathbb N}}\Delta_j^k,$
where $$\Delta_j^k=T_\gamma^{\alpha, \beta}(B_j^k)\cap [Q_j^k\times (0,\infty)]\cap [T^{(1-\eta)\alpha, (1-\eta)\beta}_\gamma(O_k^{[\bar{\eta}]})\setminus T_\gamma^{(1-\eta)\alpha, (1-\eta)\beta}(O_{k+1}^{[\bar{\eta}]})].$$

{\bf  Now, we are going to construct the atoms associated with the open set $O_k^{[\bar{\eta}]}.$} Let us define the atoms by writing $a_j^k= \frac{g \mathcal{X}_{\Delta_j^k}}{\gamma(B_j^k)^{1/q'}(\mu_j^k)^{1/q}},$ where $\mu_j^k=\iint_{\Delta_j^k}|g(y,t)|^q\, d\gamma(y)\frac{dt}{t}$ and setting $\lambda_j^k=\gamma(B_j^k)^{1/q'}(\mu_j^k)^{1/q}.$ We then obtain the following. \begin{equation}\label{ladescomposicion}g=\sum_{k\in {\mathbb Z}}\sum_{j\in {\mathbb N}}\lambda_j^k\, a_j^k.\end{equation} Observe that $a_j^k$ is a $T^{1,q}_{\alpha, \beta}(\gamma)$ atom associated with the ball $B_j^k$ that happens to be in $\mathscr{B}_{C\lambda}$ because of its radius $C d_j^k\le C \lambda m(c_{Q_j^k}).$ Therefore, (\ref{ladescomposicion}) will be the required atomic decomposition if we prove that $\sum_{k,j}\lambda_j^k\le K \|g\|_{T^{1,q}_{\alpha, \beta}(\gamma)}.$

 {\bf Proof of $\sum_{k,j}\lambda_j^k\le K \|g\|_{T^{1,q}_{\alpha, \beta}(\gamma)}.$}
Let us first estimate $\mu_j^k,$ $$\mu_j^k=\iint_{\Delta_j^k}|g(y,t)|^q\, d\gamma(y)\frac{dt}{t}\le \iint_{T_\gamma^{\alpha, \beta}(B_j^k)\cap (T_\gamma^{(1-\eta)\alpha, (1-\eta)\beta}(O_{k+1}^{[\bar{\eta}]}))^c}|g(y,t)|^q\, d\gamma(y)\frac{dt}{t}.$$ 
Observe that $(T_\gamma^{(1-\eta)\alpha, (1-\eta)\beta}(O_{k+1}^{[\bar{\eta}]}))^c=R_\gamma^{(1-\eta)\alpha, (1-\eta)\beta}(F_{k+1}^{[\bar{\eta}]}),$  
therefore we apply Lemma \ref{lema: desigualdad integral, relacion entre eta y beta}, with $H(y,t)=\mathcal{X}_{T_\gamma^{\alpha,\beta}(B_j^k)}(y,t)\, |g(y,t)|^q,$ in order to obtain that 
\begin{align*}\iint_{(T_\gamma^{(1-\eta)\alpha, (1-\eta)\beta}(O_{k+1}^{[\bar{\eta}]}))^c} &\mathcal{X}_{T^{\alpha,\beta}_\gamma (B_j^k)}(y,t)\, |g(y,t)|^q\, d\gamma(y)\frac{dt}{t}\\
&\le C \int_{B_j^k\cap F_{k+1}}(S_{q,\alpha, \beta}(g)(x))^q\, d\gamma(x) 
\end{align*}
due to the fact that $\text{supp}(S_{q,\alpha,\beta}(\mathcal{X}_{T^{\alpha, \beta}_\gamma(B_j^k)}\, g))\subseteq B_j^k$. The right side of the above inequality is bounded by $\gamma(B_j^k)\, (2^{k+1})^q.$  
Thus, $\mu_j^k\le C \gamma(B_j^k) 2^{q k}.$ Hence,
\begin{align*}
\sum_{k\in {\mathbb Z}}\sum_{j\in {\mathbb N}}\lambda_j^k & \lesssim \sum_{k\in {\mathbb Z}}\sum_{j\in {\mathbb N}}\gamma(B_j^k)^{1/q'} \gamma(B_j^k)^{1/q} 2^k\\
&\lesssim \sum_{k\in {\mathbb Z}} 2^k \sum_{j\in {\mathbb N}} \gamma(B_j^k)\\ &\lesssim \sum_{k\in {\mathbb Z}}2^k \sum_{j\in {\mathbb N}}\gamma (Q_j^k)\\ & \lesssim\sum_{k\in {\mathbb Z}}2^k \gamma (O_k^{[\bar{\eta}]})\\ &\lesssim \sum_{k\in {\mathbb Z}} 2^k \gamma(O_k)\lesssim \|S_{q,\alpha, \beta} \, g\|_{L^1(\gamma)}. 
\end{align*}
The proof of the theorem for the case $1\le q<\infty$ is concluded.

    {For $q=\infty,$ to get an atomic decomposition of a function $f$ in $T^{1,\infty}_{\alpha, \beta}(\gamma)$ we should pay attention that in this space the functions we take are continuous on $\mathbb{R}^{n+1}_+,$ thus the rough truncation procedure we have done before does not apply here. So we modify the construction done in \cite{MNP} by replacing the disjoint dyadic cubes for open balls with the same centers as the cubes' and radii half the diameter of the cubes multiply by $1+\epsilon.$ With these balls now we apply the procedure done in \cite{MNP} and for $\epsilon > 0$ small we come up with the same type of sets. The advantage now is that we get a finite open cover of $\mathbb{R}^n$ and we can construct a partition of unity associated to it.}
    
    {Let $W$ be one of these open sets from the open cover of $\mathbb{R}^n$ and $\phi$ one of the elements of the partition of unity such that $\text{supp}(\phi)\subset W.$ We call $g(y,t)=f(y,t)\phi(y)\in \mathscr{C}(\mathbb{R}^{n+1}_+).$ Now we can proceed as we have done before. First, we consider the open set $O_k=\{x\in \mathbb{R}^n: S_{\infty, \alpha, \beta}\,  g(x)>2^k\}$ and prove that it is contained in $W+\mathscr{C}_{\beta (\beta+1)}$ which is also an open $\lambda$-admissible Whitney set. In this case, we apply the Whitney lemma with balls instead of dyadic cubes. So, there exists a universal constant $C>1$ such that for the open set $O_k$ there is a sequence of $\lambda$-admissible balls $\{B_j^k=B(c_j^k,r_j^k)\}_j$ with bounded overlap verifying:
   \begin{itemize}
\item $O_k=\bigcup_j B_j^k,$
\item $C B_j^k \cap O_k^c\ne \emptyset,$ for every $j,$
\item the elements of $\{C^{-1}B_j^k\}_j$ are mutually disjoint.
   \end{itemize}
   See \cite[Theoreme 1.3, Chapitre III.1]{CW}, and \cite{artur}.}
 
   Let us prove the following inclusions.
   \begin{align}
B_j^k\times (0,\infty)\cap T_\gamma^{\alpha, \beta}(O_k)&\subset B_j^k\times (0,\infty)\cap \overset{\circ}{T}_\gamma^{\alpha/2,\beta/2}(O_k)\\ &\subset \overset{\circ}{T}_\gamma^{\alpha, \beta}((2C+3)B_j^k).\nonumber
   \end{align}
   Pick $z\in C B_j^k\cap O_k^c.$ Let $(y,t)\in B_j^k\times (0,\infty)\cap T_\gamma^{\alpha, \beta}(O_k)$ be given, then \begin{equation*}
\frac{1}{2}(\alpha t \wedge \beta m(y))<\alpha t \wedge m_\beta (y)\le \text{dist}(y,O_k^c)\le |y-z|\le (1+C)r_j^k. 
   \end{equation*}
  Thus, $(y,t)\in \overset{\circ}{T}^{\alpha/2, \beta/2}(O_k)$ and $\alpha t \wedge m_\beta(y)< 2 (C+1)r_j^k\le \text{dist}(y,((2C+3)B_j^k)^c).$
   Therefore, $(y,t)\in \overset{\circ}{T}_\gamma^{\alpha, \beta}((2C+3)B_j^k).$ 

   Let us call $\widetilde{T}_{\gamma, j}^{\alpha, \beta}(O_k)=B_j^k\times (0,\infty)\cap \overset{\circ}{T}^{\alpha/2, \beta/2}_\gamma (O_k)$ an open set and define the open set $$\Delta_j^k=\widetilde{T}_{\gamma, j}^{\alpha, \beta}(O_k)\setminus T_\gamma^{\alpha, \beta}(O_{k+1}).$$ 
   Associated to the family of open sets $\{\Delta_j^k\}_j$ we have a continuous partition of unity $\{\varphi_j^k\}_j$ with $\text{supp}(\varphi_j^k)\subset \Delta_j^k.$ It is well-known that
   \begin{equation*}
\{(y,t)\in \mathbb{R}^{n+1}_+: g(y,t)\ne 0\}\subset \bigcup_{k}T^{\alpha,\beta}_\gamma (O_k)\setminus T^{\alpha, \beta}_\gamma(O_{k+1})\subset \bigcup_{k,j}\Delta_j^k.
   \end{equation*}
   In fact, if $g(y,t)\ne 0,$ there exists $k\in {\mathbb Z}$ such that $2^k<|g(y,t)|\le 2^{k+1}.$ Then for all $x\in B(y, \alpha t\wedge m_\beta(y)),$ $(y,t)\in \Gamma_\gamma^{\alpha, \beta}(x),$ hence $S_{\infty, \alpha,\beta}\, g(x)\ge |g(y,t)|>2^{k}.$ Thus $B(y,\alpha t \wedge m_\beta(y))\subset O_k.$ Therefore,  $(y,t)\in T^{\alpha, \beta}_\gamma (O_k)$. And now, reasoning by contradiction, taking into account the other inequality, we also get $(y,t)\notin T^{\alpha,\beta}_\gamma(O_{k+1}).$ The second inclusion is immediate.
   
   Finally, it is natural to write \begin{equation*}
g(y,t)=\sum_{k, j}g (y,t)\varphi_j^k(y,t)=\sum_{k,j} 2^{k+1}\gamma((2C+3)B_j^k)a_j^k(y,t)= \sum_{k,j}\mu_j^k a_j^k(y,t),
   \end{equation*}
   being $a_j^k (y,t)=\frac{g (y,t) \varphi_j^k (y,t)} {2^{k+1}\gamma((2C+3)B_j^k)}$ {{and $\mu_j^k=2^{k+1}\gamma((2C+3)B_j^k). $}}}

   Let us see that $a_j^k$ is a $T_{\alpha, \beta}^{1,\infty}(\gamma)$ $\delta$-atom, for some $\delta>0.$
According to the definition of $a_j^k,$ $\text{supp}(a_j^k)\subset T^{\alpha, \beta}_\gamma( (2C+3)B_j^k),$ $\|a_j^k\|_\infty\le \frac{\|g \mathcal{X}_{\Delta_j^k}\|_\infty}{2^{k+1} \gamma((2C+3)B_j^k)}\le \gamma((2C+3)B_j^k)^{-1},$ and it is continuous. 

Let us prove that $\displaystyle \sum_{k,j}|\mu_j^k|\le C \|g\|_{T^{1,\infty}_{\alpha, \beta}(\gamma)}.$ Indeed, \begin{align*}
\sum_{k\in{\mathbb Z}} 2^{k+1}\sum_{j=1}^\infty \gamma(C(2C+3)C^{-1}B_j^k))&\lesssim \sum_{k\in {\mathbb Z}}2^{k+1} \sum_{j=1}^\infty
\gamma(C^{-1}B_j^k)\\&\lesssim \sum_{k\in {\mathbb Z}}2^k \gamma(O_k)\\ &\lesssim \|S_{\infty,\alpha, \beta}\, g\|_{L^1(\gamma)}.\end{align*}

The first inequality is due to the fact that $\gamma$ doubles on admissible balls. For the second inequality, we use that the balls $\{C^{-1}B_j^k\}_j$ are mutually disjoint and its union is contained in $O_k.$ With this, we end the proof of Theorem \ref{teo: descomposicion atomica}.
 
\subsection{Proof of Theorem \ref{coneccion1} }\label{edison1}

Before proving Theorem \ref{coneccion1}, we shall prove the existence of
an 
operator that maps measurable functions on $\mathbb{R}^{n+1}_+$ to measurable functions on $\mathbb{R}^n.$
For that, 
we fix a function $\phi\in \mathscr{C}^1_c(\mathbb{R}^n)$ that satisfies (see \cite{cms}):
\begin{itemize}
\item $\phi$ has compact support in the unit ball,
\item $|\phi(x)|\le M,$ $|\nabla \phi (x)|\le M,$ for some positive constant $M$ and $\hbox{ for any } x\in{\mathbb R}^n,$ 
\item $\int_{{\mathbb R}^n}\phi(x)\, dx=0,$
\item $|\hat{\phi}(\xi)|\le M \left(|\xi|\wedge \frac{1}{|\xi|}\right).$
\end{itemize}
We shall consider then the operator $\pi_\phi$ defined as
\begin{align}\label{conexionH1}
\pi_\phi f(x) &= e^{|x|^2}\int_0^\infty \mathcal{X}_D(\cdot, t) f(\cdot, t)e^{-|\cdot|^2}\ast \phi_{t}(x)\, \frac{dt}{t},
\end{align}
where $D=\{(y,t)\in \mathbb{R}^{n+1}_+: t<m(y)\}$ and $\phi_t (x) = \frac{1}{t^n} \phi(\frac{x}{t}).$
So,\begin{align*} \pi_\phi f (x)
&= e^{|x|^2}\int_0^1 \int_{\{y\in \mathbb{R}^n: t<m(y)\}} f(y,t)\phi_t(x-y) d\gamma (y)\frac{dt}{t}.\nonumber
\end{align*} For instance, this formula makes sense for the space $\mathscr{C}_c({\mathbb R}^{n+1}_+).$ In fact, let $f\in \mathscr{C}_c({\mathbb R}^{n+1}_+)$ be given, then there exist $R\ge 1$ and $0<a<b<\infty,$ such that $\text{supp}(f)\subseteq K:=\bar{B}(0,R)\times [a,b],$ then
\begin{align*}
|\pi_\phi f (x)|&\le e^{|x|^2} \int_a^{1\wedge b}\int_{|y|\le R}|f(y,t)|\left|\phi\left(\frac{x-y}{t}\right)\right|\frac{1}{t^n}\, e^{-|y|^2} dy\, \frac{dt}{t} \\ &\le e^{|x|^2}\|f\|_\infty M \int_a^\infty \int_{|y|\le R}\frac{1}{t^{n+1}}\, dy \, dt<\infty.
\end{align*}

\noindent {\bf Proof of (i)}

\begin{itemize}
\item  Step 1: We shall prove that $\pi_\phi$ maps a $T^{1,2}(\gamma)$ $\delta$-atom to a bounded multiple of an $H^1(\gamma)$ atom. 

To see this, suppose $a(y,t)$ is a $T^{1,2}(\gamma)$ $\delta$-atom associated with a ball $B,$ i.e. 
$B$ is such that $r_B\le \delta m(c_B).$ We will prove that $\pi_\phi (a) $
has compact support, $\gamma$-average 0 and $L^2 (\gamma)$
norm bounded by $\frac{C}{\gamma(B)^{1/2}} $. In fact, 
\begin{itemize}
\item Bounded support

Being $a(y,t)$ an atom in $T^{1,2}(\gamma)$, then it is supported in a Gaussian tent 
$ 
T_\gamma(B)$ 
Thus, $\displaystyle \pi_\phi a (x)=e^{|x|^2}
\iint_{T^D_\gamma(B)} a(y,t) \phi_t(x-y)d\gamma(y)\frac{dt}{t}.$  

Since $\phi$ is supported in $|x|\le 1,$ it follows that $\pi_\phi (a)$ is supported in a ball $B^*,$ having the same center as $B,$ but twice its radius. Indeed, for $|x-c_B|\ge 2r_B$ and $(y,t)\in T_\gamma^D(B)$ then
$$|x-y|\ge |x-c_B|-|y-c_B|\ge 2r_B-r_B+t=r_B+t.$$ Thus $\left|\frac{x-y}{t}\right|\ge 1+\frac{r_B}{t}>1.$ Hence $\phi\left(\frac{x-y}{t}\right)=0$ and so $\pi_\phi a (x)=0.$ Thus, 
$\text{supp}(\pi_\phi a)\subseteq B^*.$

\item $\gamma$-average zero 

$\int_{{\mathbb R}^n}\pi_\phi a (x)\, d\gamma(x)=0$ since $\int_{{\mathbb R}^n} \phi (x)\, dx=0.$ Indeed, by using Fubini's theorem we obtain that
\begin{align*}
\int_{{\mathbb R}^n} \pi_\phi a (x)e^{-|x|^2}\, dx& = \int_{{\mathbb R}^n}\iint_{T^D_\gamma(B)} a(y,t)\phi_{t}(x-y)\, d\gamma(y)\, \frac{dt}{t}\, dx
\\ &= \iint_{T^D_\gamma(B)}a(y,t)\int_{{\mathbb R}^n}\phi_{t}(x-y)\, dx\, d\gamma(y)\, \frac{dt}{t}\\ &=\iint_{T^D_\gamma(B)}a(y,t)\left(\int_{{\mathbb R}^n}\phi(z)\, dz\right)\, d\gamma(y)\, \frac{dt}{t}=0
\end{align*}

\item $L^2(\gamma)$ norm boundedness

Finally, let us prove that 
\begin{align*}
\int_{{\mathbb R}^n} |\pi_\phi a (x)|^2 d\gamma (x)&\le C \iint_{T_\gamma(B)}|a(y,t)|^2 \, d\gamma(y)\, \frac{dt}{t}.
\end{align*}
So, $\|\pi_\phi a\|_{L^2(\gamma)}\le \frac{C}{\gamma(B)^{1/2}}.$

For $g\in C_c^\infty ({\mathbb R}^n)$ and defining $\widetilde{\phi}(z)=\phi(-z),$ we have
\begin{align*}
\left|\int_{{\mathbb R}^n}\pi_\phi a (x)\, g(x)\, d\gamma(x)\right|&= \left|\int_{{\mathbb R}^n}\iint_{T^D_\gamma(B)}  a(y,t) \phi_{t} (x-y)\, d\gamma(y)\frac{dt}{t}g(x)\, dx\right|\\ & =\left| \iint_{T^D_\gamma(B)}a(y,t) \int_{{\mathbb R}^n}\phi_{t}(x-y)g(x)dx\ d\gamma(y)\frac{dt}{t}\right|\\ &= \left|\iint_{T^D_\gamma(B)} a(y,t)\  g\ast \widetilde{\phi}_{t}(y)\, d\gamma(y)\frac{dt}{t}\right|\\ &\le \iint_{{\mathbb R}^{n+1}_+} |F(y,t)|| G(g)(y,t)|\, d\gamma(y)\frac{dt}{t}
\end{align*}
with $F(y,t)=\mathcal{X}_{T_\gamma(B)}(y,t)\,  a(y,t)$ and $G(g)(y,t)=\mathcal{X}_{D}(y,t)\, g\ast \widetilde{\phi}_{t}(y). $ Then
\begin{align*}
& \iint_{{\mathbb R}^{n+1}_+}   |F(y,t)| |G(g)(y,t)|\, d\gamma(y)\frac{dt}{t}
\\ &= \iint_{{\mathbb R}^{n+1}_+}\frac{|F(y,t)|  {|G(g)(y,t)|}  }{\gamma (B(y,t\wedge m(y)))}  \int_{{\mathbb R}^n}\mathcal{X}_{\Gamma_\gamma(x)}(y,t) d\gamma(x)\, d\gamma(y)\frac{dt}{t}\\ 
& = \int_{{\mathbb R}^n}\left( \iint_{{\mathbb R}^{n+1}_+}\mathcal{X}_{\Gamma_\gamma(x)} (y,t)\frac{|F(y,t)||G(g)(y,t)|}{\gamma (B(y,t\wedge m(y)))}\right.  \left.    d\gamma(y)\frac{dt}{t}\right) d\gamma (x).
\end{align*}
Now, by applying Schwartz's inequality twice we get 
\begin{align*}
\left|\int_{{\mathbb R}^n}\pi_\phi a (x)\, g(x)\, d\gamma(x)\right|& \le \int_{{\mathbb R}^n}S_{2}F(x)S_{2}G(g)(x) d\gamma(x)\\ &\le \|S_{2}F\|_{L^2(\gamma)}\|S_{2}G(g)\|_{L^2(\gamma)}.
\end{align*}
It is immediate to see that $$\|S_{2}F\|_{L^2(\gamma)}= \|F\|_{T^{1,2}(\gamma)}=\left(\iint_{T_\gamma (B)}|a(y,t)|^2\, d\gamma(y)\frac{dt}{t}\right)^{1/2}\le \frac{1}{\gamma(B)^{1/2}}.$$

The hardest part to prove is the following inequality
$$\|S_{2}G(g)\|_{L^2(\gamma)}\le A \| g \|_{L^2 (\gamma)}.$$ But this is a consequence of the standard theory of singular integrals in a harmonic analysis context.  Indeed,
let $\{B_j\}$ be a countable family of admissible balls covering $\mathbb{R}^n$ (see \cite{dalsco}), i.e., we set $B_j=B(c_{B_j},r_{B_j})$ the ball of center $c_{B_j}$ and radius $r_{B_j}$ with $r_{B_j}\le m(c_{B_j})$ and call $B^{**}_j=4B_j$ then the family $\{B^{**}_j\}$ has bounded overlap. 
Hence, 
\begin{align*}
\|S_{2}G(g)\|_{L^2(\gamma)}^2&=\int_{\mathbb{R}^n}(S_{2} G(g)(x))^2d\gamma(x)\\ &=\int_{\mathbb{R}^n}\iint_{\Gamma_\gamma(x)}\frac{|G(g) (y,t)|^2}{\gamma(B(y,t\wedge m(y)))}d\gamma(y)\frac{dt}{t}d\gamma(x)\\
&= \int_{\mathbb{R}^n}\iint_{{\Gamma}^D_\gamma(x)}\frac{|g\ast \widetilde{\phi}_t(y)|^2}{\gamma(B(y,t))}d\gamma(y)\frac{dt}{t}d\gamma(x)\\
&=\iint_{D}\frac{|g\ast \widetilde{\phi}_t(y)|^2}{\gamma(B(y,t))}\gamma(B(y,t))d\gamma(y)\frac{dt}{t}\\
&= \int_{\mathbb{R}^n}\int_0^{m(y)}|g\ast \widetilde{\phi}_t(y)|^2\frac{dt}{t}d\gamma(y)\\
&\le \sum_j \int_{B_j}\int_0^{m(y)}|g\ast \widetilde{\phi}_t(y)|^2\frac{dt}{t}d\gamma(y)\\ &\le \sum_j e^{-|c_{B_j}|^2}\int_{B_j}\int_0^{2m(c_{B_j})}|g\ast \widetilde{\phi}_t(y)|^2\frac{dt}{t}dy\\ &= \sum_j e^{-|c_{B_j}|^2}\int_{B_j}\int_0^{2m(c_{B_j})}|\mathcal{X}_{B^{**}_j}g\ast \widetilde{\phi}_t(y)|^2\frac{dt}{t}dy
\\&\le \sum_j e^{-|c_{B_j}|^2}\int_{B_j}\int_0^\infty|\mathcal{X}_{B^{**}_j}g\ast \widetilde{\phi}_t(y)|^2\frac{dt}{t}dy\\ &\le\sum_j e^{-|c_{B_j}|^2}\int_0^\infty \int_{\mathbb{R}^n}|\widehat{\mathcal{X}_{B_j^{**}} g}(\xi)|^2 |\widehat{\widetilde{\phi}}(t\xi)|^2\, d\xi \frac{dt}{t}
\end{align*}
\begin{align*}
&\le \left(\sup_{\xi}\int_0^\infty |\widehat{\widetilde{\phi}}(t\xi)|^2\frac{dt}{t}\right) \sum_j e^{-|c_{B_j}|^2} \int_{\mathbb{R}^n}|\widehat{\mathcal{X}_{B_j^{**}} g}(\xi)|^2 d\xi\\ &=  \left(\sup_{\xi}\int_0^\infty |\widehat{\widetilde{\phi}}(t\xi)|^2\frac{dt}{t}\right) \sum_j e^{-|c_{B_j}|^2} \int_{B_j^{**}}| g(y)|^2 dy\\  & \le  M^2 \widetilde{C} \|g\|_{L^2(\gamma)}^2.
\end{align*}

\end{itemize}

\item Step 2: Let $f\in T^{1,2}(\gamma)$ be given. By Theorem \ref{teo: descomposicion atomica}, there exist a sequence of scalars $\{\lambda_m\}$ and a sequence of $\delta$-atoms $\{a_m\}$ such that $f=\sum_m \lambda_m a_m$ with $\sum_m |\lambda_m|\le A \|f\|_{T^{1,2}(\gamma)}.$ Then the series $\sum_m \lambda_m \pi_\phi (a_m)$ converges in $L^1(\gamma)$ since $\|\pi_\phi(a_m)\|_{L^1(\gamma)}\le C$ for all $m.$ Therefore, by defining $\pi_\phi(f)=\sum_m C\lambda_m  C^{-1}\pi_\phi (a_m)\in H^1(\gamma),$ we have $\pi_\phi:T^{1,2}(\gamma)\to H^1(\gamma)$ a linear bounded operator. In fact, $\|\pi_\phi (f)\|_{H^1(\gamma)}\le \sum_m C|\lambda_m|\le C A \|f\|_{T^{1,2}(\gamma)}.$
Thus, the first assertion is proved.

\end{itemize}

\noindent {\bf Proof of (ii)} We calculate the $L^p(\gamma)$ norm using duality. Let $f\in T^{p,2}(\gamma)$ and $g\in L^{p'}(\gamma)$ with $p'$ the conjugate exponent to $p.$ We proceed as in case $L^2$:
\begin{align*}
&\left|\int_{\mathbb{R}^{n}}  \pi_\phi f(x) g(x)\, d\gamma(x)\right|\\&=\left|\int_{\mathbb{R}^n}\int_0^1 \int_{\{y\in \mathbb{R}^n: t<m(y)\}}f(y,t) \, g\ast \widetilde{\phi}_t(x-y)d\gamma(y)\frac{dt}{t}\,  g(x)\, dx\right|\\ & \le \int_{\mathbb{R}^{n}}S_{2} f (x)\, S_{2}G(g)(x)\, d\gamma(x)\\ &\le \|S_{2}f\|_{L^p(\gamma)}\, \|S_{2}G(g)\|_{L^{p'}(\gamma)}=\|f\|_{T^{p,2}(\gamma)}\, \|S_{2}G(g)\|_{L^{p'}(\gamma)}.
\end{align*}  
In order to prove that $$\|S_{2}G(g)\|_{L^{p'}(\gamma)}\le C \|g\|_{L^{p'}(\gamma)},$$ we proceed  by localizing the $L^{p'}(\gamma)$ norm with the family of admissible balls $\{B_j\}$ and apply the standard theory of vector-valued singular integral in Harmonic Analysis context (see \cite{stein}), to obtain
\begin{align*}
\int_{\mathbb{R}^n} |S_{2} G(g)(x)|^{p'}d\gamma(x)&=\int_{\mathbb{R}^n}\left(\iint_{\Gamma_\gamma^D(x)}\frac{|g\ast \widetilde{\phi}_t(y)|^2}{\gamma(B(y,t))}d\gamma(y)\frac{dt}{t}\right)^{\frac{p'}{2}}d\gamma(x)\\ &\lesssim  \sum_{j}\int_{B_j} \left( \iint_{\Gamma_\gamma^D(x) }|g\ast \widetilde{\phi}_t (y)|^2 dy \frac{dt}{t^{n+1}}\right)^{\frac{p'}{2}}d\gamma(x)
\\ & \lesssim \sum_{j} e^{-|c_{B_j}|^2} \int_{B_j}\left(\iint_{\Gamma(x)}|\mathcal{X}_{B_j^{**}} g\ast \widetilde{\phi}_t (y)|^2 dy \frac{dt}{t^{n+1}}\right)^{\frac{p'}{2}}dx\\ &\lesssim \sum_{j} e^{-|c_{B_j}|^2} \int_{B_j^{**}}|g(x)|^{p'}dx\lesssim 
\|g\|_{L^{p'}(\gamma)}^{p'},
\end{align*}
where $\Gamma(x)=\{(y,t)\in \mathbb{R}^{n+1}_+: |y-x|<t\}$ is the classical cone in the upper half hyperplane and we have used the bounded overlap of the family $\{B_j^{**}\},$ to reconstruct the $L^{p'}(\gamma)$ norm of $g.$

This ends the proof of the theorem's second assertion.

\noindent {\bf Proof of (iii)}.
For the third, we use the fact that the dual of $H^1(\gamma)$ is the Gaussian space $BMO(\gamma),$ see \cite{MM}. Then for $f\in T^{\infty, 2}(\gamma)$ it will suffice to see that $\pi_\phi(f)$ can be paired with $H^1(\gamma)$, that is, it suffices to show that for $g\in H^1(\gamma)$ the following inequality holds.
\begin{equation} 
\left|\int_{{\mathbb R}^n} \pi_\phi (f)(x)\, g(x)\, d\gamma(x)\right|\le
C  \|g\|_{H^1(\gamma)}. 
\end{equation}

Thus, let $g\in H^1(\gamma)$. Let us estimate $\int_{{\mathbb R}^{n}} \pi_\phi(f)(x)\, g(x)\, d\gamma(x)$ arguing as we have done above.
\begin{align*}\label{bmo}
\int_{{\mathbb R}^n} \pi_\phi (f)(x)\, g(x)\, d\gamma(x)&=\iint_{D} f(y,t)\ g\ast \widetilde{\phi}_t(y)\, d\gamma(y)\frac{dt}{t}\nonumber \\ &= \iint_{\mathbb{R}^{n+1}_+} f(y,t)\, G(g)(y,t)\, d\gamma(y)\frac{dt}{t},
\end{align*}
being $G(g)(y,t)=\mathcal{X}_{D}(y,t)\, \widetilde{\phi}_t\ast g (y).$ From the inequality (\ref{dual1inftygral}) we obtain
\begin{align*}
\left|\int_{\mathbb{R}^n}\pi_\phi f(x)\, g(x)\, d\gamma(x)\right|&\le C \int_{\mathbb{R}^n} C_2 f (x)\, S_2 G(g) (x)\, d\gamma(x)\\ &\le C \|C_2 f\|_{L^\infty(\gamma)}\, \|S_2 G(g)\|_{L^1(\gamma)} \\ &= C \|f\|_{T^{\infty,2}(\gamma)}\, \|S_2 G(g)\|_{L^{1}(\gamma)}.
\end{align*}
It remains to prove that
\begin{equation}\label{continuidadH1}\|S_2 G(g)\|_{L^{1}(\gamma)}\le K \, \|g\|_{H^1(\gamma)}.\end{equation} This boundedness follows once it is shown that 
\begin{enumerate}
\item There exists a constant $K>0,$ such that $\|S_2G(a)\|_{L^1(\gamma)}\le K,$ for every $(1,\infty)$-atom $a$ in $H^1(\gamma).$ \label{boulder1}

\item 
$S_2G$ satisfies a weak-type $(1,1)$ inequality with respect to $\gamma$.\label{boulder2}
\end{enumerate}

Assume that \ref{boulder1} and \ref{boulder2} have been proven. Take
$g\in H^1(\gamma)$, then $g=\sum_{j=1}^\infty\lambda_j\, a_j$ with $a_j$ a $(1,\infty)$-atom and $\sum_{j=1}^\infty|\lambda_j|<\infty,$ see \cite{MMS}. It is known that the convergence of the series to $g$ is in $L^1(\gamma),$ then the sequence $\{S_2G (\sum_{j=1}^k \lambda_j a_j)\}_k$ converges to $S_2G(g)$ in weak-$L^1(\gamma).$ Hence, there exists an increasing function $l:{\mathbb N}\to {\mathbb N}$ such that $$S_2G(g)(x)=\lim_{k\to \infty} S_2G\left(\sum_{j=1}^{l(k)}\lambda_j \, a_j\right)(x), \ \ \text{a.e.}\ x\in \mathbb{R}^n.$$ 
Due to the sublinearity of $S_2G,$ we have $S_2G(g)(x)\le \sum_{j=1}^\infty |\lambda_j|\, S_2G(a_j)(x)$ a.e. $x\in \mathbb{R}^n.$ Hence, $$\|S_2G(g)\|_{L^1(\gamma)}\le \sum_{j=1}^\infty |\lambda_j|\, \|S_2G(a_j)\|_{L^1(\gamma)}\le K \sum_{j=1}^\infty |\lambda_j| .$$ 
For every atomic decomposition in which $g$ is written, so inequality (\ref{continuidadH1}) follows.

\

\noindent{\bf Proof of \ref{boulder1}:}
First a remark
\begin{remark}\label{soporteatomico}
Let $a$ be a $(1,\infty)$-atom, then we have $\text{supp} (S_2G(a))\subseteq \bar{B}(c_B, 9 m(c_B)),$ being $B=B(c_B,r_B)$ the admissible ball containing $\text{supp}(a).$
\end{remark}
In fact, for $x\notin \bar{B}(c_B, 9 m(c_B))$ such that $S_2G(a)(x)<\infty,$ we will see that for every $(y,t)\in D$ such that $\widetilde{\phi}_t\ast a (y)\ne 0,$ 
then $(y,t)\notin \Gamma_\gamma (x),$ that is, $|x-y|\ge t.$ Indeed, from the assumption on $y,$ the ball $B(y,t)$ must intersect $B.$ Then, take $z\in B$ such that $|y-z|<t.$ From this fact and $(y,t)\in D,$ we obtain, taking into account lemma \ref{comparacion}, $m(z)\le 2 m(c_B),$ $m(y)\le 2 m(z)\le 4 m(c_B),$ and $t<m(y)\le 4 m(c_B).$ Thus,
$$|x-y|\ge |x-c_B|-|y-c_B|> 9 m(c_B)-(r_B+|y-z|)>4 m (c_B)> t.$$ Therefore, $S_2 G(a)(x)=0,$ a.e. $x\notin \bar{B}(c_B, 9 m(c_B)).$  Hence, $\text{supp}(S_2G(a))\subseteq \bar{B}(c_B,9 m(c_B)),$ and the remark \ref{soporteatomico} follows.   

Now, let us prove that there exists a constant $C>0,$ such that $\|S_2 G(a)\|_{L^1(\gamma)}\le C.$ In fact,
\begin{align*}\int_{\mathbb{R}^n}S_2G(a)(x)\, d\gamma(x)&=\int_{2B} S_2G(a)(x)\, d\gamma(x)+ \int_{B(c_B, 9 m(c_B))\setminus 2B} S_2G(a)(x)\, d\gamma(x)\\
&=(I)+(II).\end{align*}
For estimating $(I)$ we use Schwarz's inequality followed by the boundedness of the operator $S_2G$ on $L^{2}(\gamma)$ (see item (i) of this theorem) and the fact that $\|a\|_{L^2(\gamma)}\le \gamma(B)^{-1/2},$ and $\gamma(2B)\le C \gamma(B).$ Let us estimate the term $(II).$  

Let us consider the area function 
 $$S_2G(a)(x)=\left(\iint_{\Gamma_\gamma^D (x)}|a\ast \widetilde{\phi}_t(y)|^2\, \frac{d\gamma(y)}{\gamma(B(y,t))}\frac{dt}{t}\right)^{1/2}.$$ Then, $S_2G(a)(x)\lesssim \left(\iint_{\Gamma(x)\cap D}|a\ast\widetilde{\phi}_t(y)|^2\frac{dy}{t^n}\frac{dt}{t}\right)^{1/2},$ being $\Gamma (x)=\{(y,t)\in \mathbb{R}^{n+1}_+: |y-x|<t\}.$

 By defining $\Gamma:=\Gamma (0),$ we can write the last area function as a vector-valued convolution operator, i.e.,
$$S_2G(a)(x)\lesssim\left|a\ast \widetilde{\phi}_t (x-y) \right|_{L^2(\Gamma^D, \frac{dy}{t^n}\frac{dt}{t})},$$ with $\Gamma^D=\Gamma\cap D.$ Let us estimate $a\ast \widetilde{\phi}_t(x-y).$ We have
$$a\ast \widetilde{\phi}_t(x-y)=\int (\widetilde{\phi}_t(x-y-z)-\widetilde{\phi}_t(x-y-c_B))\, a(z)\, dz++ \widetilde{\phi}_t(x-y-c_B)\int a(z)\, dz.$$ 
We may assume $t>\text{dist}(x-y, B)$ so that $t>\frac{|x-y-c_B|}{2},$ since otherwise $a\ast \widetilde{\phi}_t(x-y)$ will be zero. So, the set $\{(y,t)\in \Gamma^D: t>\frac{|x-y-c_B|}{2}\}\subseteq \{(y,t)\in \Gamma^D: t>\frac{|x-c_B|}{3}\}=: W.$

Thus,
\begin{align*}
&|a\ast \widetilde{\phi}_t(x-y)|_{L^2(\Gamma^D, \frac{dy}{t^n}\frac{dt}{t})} \le |a\ast \widetilde{\phi}_t(x-y)|_{L^2(W, \frac{dy}{t^n}\frac{dt}{t})}\\ & \le \int |\widetilde{\phi}_t(x-y-c_B-(z-c_B))-\widetilde{\phi}_t(x-y-c_B)|_{L^2(W, \frac{dy}{t^n}\frac{dt}{t})}\ |a(z)|\, dz\\ &\ \ \ + |\widetilde{\phi}_t(x-y-c_B)|_{L^2(W, \frac{dy}{t^n}\frac{dt}{t})}\ \left|\int a(z)\, dz\right|.
\end{align*}
We will prove that $$|\widetilde{\phi}_t(x-y-c_B-(z-c_B))-\widetilde{\phi}_t(x-y-c_B)|_{L^2(W, \frac{dy}{t^n}\frac{dt}{t})}\lesssim \frac{|z-c_B|}{|x-c_B|^{n+1}},$$ and $$|\widetilde{\phi}_t(x-y-c_B)|_{L^2(W, \frac{dy}{t^n}\frac{dt}{t})}\lesssim \frac{1}{|x-c_B|^n}.$$
Indeed,
\begin{align*}
|\widetilde{\phi}_t&(x-y-c_B-(z-c_B))-\widetilde{\phi}_t(x-y-c_B)|^2_{L^2(W, \frac{dy}{t^n}\frac{dt}{t})}\\ &=\int_{t>\frac{|x-c_B|}{3}}\int_{|y|<t}|\widetilde{\phi}_t(x-c_B-(z-c_B)-y)-\widetilde{\phi}_t(x-c_B-y)|^2\, \frac{dy}{t^n}\frac{dt}{t}\\ &=\int_{|w|<1}\int_{t>\frac{|x-c_B|}{3}} \left|\frac{1}{t^n}\left(\widetilde{\phi}\left(\frac{x-c_B-(z-c_B)}{t}-w\right)-\widetilde{\phi}\left(\frac{x-c_B}{t}-w\right)\right)\right|^2\, \frac{dt}{t}dw
\end{align*}
Now, we apply the mean value theorem to the function $\widetilde{\phi}$ and use the assumption $|\nabla \phi(x)|\le M$ for all $x\in \mathbb{R}^n$ to obtain that the last term of the above inequality is bounded by
$$C \int_{t>\frac{|x-c_B|}{3}}\frac{dt}{t^{2(n+1)+1}}|z-c_B|^2\lesssim \left(\frac{|z-c_B|}{|x-c_B|^{n+1}}\right)^2.$$
On the other hand,
\begin{align*}|\widetilde{\phi}_t(x-y-c_B)|^2_{L^2(W, \frac{dy}{t^n}\frac{dt}{t})}&=\int_{t>\frac{|x-c_B|}{3}}\int_{|y|<t}\left|\frac{1}{t^n}\widetilde{\phi}\left(\frac{x-c_B}{t}-\frac{y}{t}\right)\right|^2\, \frac{dy}{t^n}\frac{dt}{t}\\&=\int_{t>\frac{|x-c_B|}{3}}\int_{|w|<1}\frac{1}{t^{2n}}\left|\widetilde{\phi}\left(\frac{x-c_B}{t}-w\right)\right|^2\, dw\frac{dt}{t}\\ &\lesssim \int_{t>\frac{|x-c_B|}{3}}\frac{dt}{t^{2n+1}}\lesssim \frac{1}{|x-c_B|^{2n}}. \end{align*}
Also, taking into account that $\int a(z)\, e^{-|z|^2}\, dz=0,$\begin{align*}
\int a(z)\, dz&= e^{|c_B|^2}\int a(z)(e^{-|c_B|^2}-e^{-|z|^2})\, dz\\ & = e^{|c_B|^2}\int a(z) (e^{(z-c_B)\cdot(z+c_B)}-1)d\gamma(z),
\end{align*}
so, $$\left|\int a(z)\, dz\right|\lesssim e^{|c_B|^2}\int_B |z-c_B|(1+|c_B|)|a(z)|d\gamma(z)\lesssim r_B(1+|c_B|) e^{|c_B|^2}.$$
Bringing all things together, we have
\begin{align*}|a\ast \widetilde{\phi}_t(x-y)|_{L^2(\Gamma^D, \frac{dy}{t^n}\frac{dt}{t})}&\lesssim \frac{1}{|x-c_B|^{n+1}}\int_B |z-c_B|\, |a(z)|\, dz\\
&\;\;\;+ \frac{1}{|x-c_B|^n}r_B(1+|c_B|)e^{|c_B|^2}\\ &\lesssim \frac{r_B\, e^{|c_B|^2}}{|x-c_B|^{n+1}}+\frac{r_B(1+|c_B|)\, e^{|c_B|^2}}{|x-c_B|^n}.\end{align*}
Hence, 
$$(II)=\int_{B(c_B, 9 m(c_B))\setminus 2B}S_2G(a)(x)\, d\gamma(x)\lesssim 1+r_B(1+|c_B|)\log\left(\frac{9m(c_B)}{r_B}\right)\le C.$$

\

\noindent{\bf Proof of \ref{boulder2}:}
Now, we are going to prove that $S_2G$ is weak-type $(1,1)$ continuous with respect to $\gamma,$ i.e.,  there exists $C>0$ such that for every $f\in L^1(\gamma)$ and $\lambda>0,$ $\gamma\{x\in \mathbb{R}^n: S_2G(f)(x)>\lambda\}\le C\|f\|_{L^1(\gamma)}/\lambda.$ We proceed as we did before in the calculation of the $L^{p'}(\gamma)$ norm of $S_2G.$ Take $f\in L^1(\gamma),$ then we have for all $x\in B_j,$ $S_2G(f)(x)=S_2G(\mathcal{X}_{B_j^{**}}f)(x)$ and $S_2G(f)(x)\lesssim |f\ast\widetilde{\phi}_t(x-y)|_{L^2(\Gamma, \frac{dy}{t^n}\frac{dt}{t})}.$ Thus
\begin{align*}
\gamma\{x\in\mathbb{R}^n: & S_2G(f)(x)>\lambda\} \le \sum_j \gamma\{x\in B_j: S_2G(\mathcal{X}_{B^{**}_j}f)(x)>\lambda\}
\\ & \le \sum_j \gamma\{x\in B_j: |\mathcal{X}_{B_j^{**}}f\ast \widetilde{\phi}_t (x-y)|_{L^2(\Gamma, \frac{dy}{t^n}\frac{dt}{t})}>\lambda/C\}
\\ & \lesssim \sum_j e^{-|c_{B_j}|^2} |\{x\in \mathbb{R}^n: |\mathcal{X}_{B_j^{**}}f\ast \widetilde{\phi}_t (x-y)|_{L^2(\Gamma, \frac{dy}{t^n}\frac{dt}{t})}>\lambda/C\}|.
\end{align*}
According to \cite{stein}, the Area function $|f\ast \widetilde{\phi}_t(x-y)|_{L^2(\Gamma, \frac{dy}{t^n}\frac{dt}{t})}$ is weak-type $(1,1)$ continuous with respect to Lebesgue measure. So,
\begin{align*}
\gamma\{x\in\mathbb{R}^n: S_2G(f)(x)>\lambda\}&\lesssim \sum_j \frac{e^{-|c_{B_j}|^2}}{\lambda}\, \int_{B^{**}_j}|f(y)|\, dy\\ & \lesssim \frac{1}{\lambda}\int_{\mathbb{R}^n}|f(y)|\, d\gamma(y).
\end{align*}
This ends the proof of the theorem's third assertion.
\section{Proofs of Section \ref{carleson}}
\subsection{Proof of Theorem \ref{car1}}\label{car1proof}
{{ Let $f\in T^{1,\infty}_{\alpha, \beta}(\gamma)$ be given. Then \begin{equation}\label{acotacionmu}
    \iint_{\mathbb{R}^{n+1}_+}|f(y,t)|\, d|\mu| (y,t)\lesssim_{n,\delta} \|\mu\|_C\ \|f\|_{T_{\alpha,\beta}^{1,\infty}(\gamma)}
\end{equation} 
Indeed, according to the proof of Theorem \ref{teo: descomposicion atomica} for the case $q=\infty,$ there exists a finite open covering of $\mathbb{R}^n$, $\{W_j\}_{j=1}^N$ and a continuous partition of unity $\{\phi_j\}_{j=1}^N$ with $\phi_j\in \mathscr{C}_c(\mathbb{R}^n)$ and $\text{supp}(\phi_j)\subset W_j$ such that $$f(y,t)=\sum_{j=1}^N f(y,t)\, \phi_j(y)=:\sum_{j=1}^N f_j(y,t)$$ with $f_j\in \mathscr{C}(\mathbb{R}^{n+1}_+).$ Thus,
$$\iint_{\mathbb{R}^{n+1}_+}|f(y,t)|\, d|\mu|(y,t)\le \sum_{j=1}^N \iint_{\mathbb{R}^{n+1}_+}|f_j(y,t)|\, d|\mu|(y,t).$$
Now, let us compare the open sets $\{(y,t)\in \mathbb{R}^{n+1}_+: |f_j(y,t)|>\kappa \}$ and $\Omega_\kappa^j:=\{x\in \mathbb{R}^n: S_{\infty, \alpha, \beta}(f_j)(x)>\kappa\}.$ For every $(y,t)$ such that $|f_j(y,t)|>\kappa,$ then $B(y,\alpha t\wedge m_\beta(y))\subset \Omega_\kappa^j,$ hence $\text{dist}(y,(\Omega_\kappa^j)^c)\ge \alpha t\wedge m_\beta (y),$ i.e., $(y,t)\in T^{\alpha, \beta}_\gamma(\Omega_\kappa^j).$ This implies $\{(y,t)\in \mathbb{R}^{n+1}_+: |f_j(y,t)|>\kappa\}\subseteq T_\gamma^{\alpha, \beta}(\Omega_\kappa^j).$ It was proved in Theorem \ref{teo: descomposicion atomica}  that $\Omega_\kappa^j$ is contained in $W_j+\mathscr{C}_{\beta(1+\beta)}$ which is an open Whitney $\lambda$-set. Then, from the Whitney lemma applied to $\Omega_\kappa^j$, see \cite{Stein-Singular integrals and differentiability properties}, we obtain a disjoint sequence of dyadic cubes $\{Q_i^j\}_{i\in {\mathbb N}}$ such that $\Omega_\kappa^j=\bigcup_{i\in {\mathbb N}}Q_i^j$ and $\text{diam}(Q_i^j)=:d_i^j\le \text{dist}(Q_i^j,(\Omega_\kappa^j)^c)\le 4 d_i^j.$ Let us call $B_i^j$ an open ball centered at $c_{Q_i^j},$ the center of $Q_i^j,$ and radius $C>0$ times $d_i^j,$ where $C$ will be chosen large enough so that $T^{\alpha, \beta}_\gamma(\Omega_\kappa^j)\subseteq \bigcup_{i\in {\mathbb N}}T_\gamma^{\alpha, \beta}(B_i^j).$ In fact, take $(y,t)\in T^{\alpha, \beta}_\gamma(\Omega_\kappa^j),$ then 
$y\in \Omega_\kappa^j$ and there exists $i\in {\mathbb N}$ such that $y\in Q_i^j.$ Thus $\alpha t\wedge m_\beta(y)\le \text{dist}(y,(\Omega_\kappa^j)^c)\le \text{dist}(Q_i^j, (\Omega_\kappa^j)^c)+d_i^j\le 5 d_i^j.$ On the other hand,  $\text{dist}(y, (B_i^j)^c)\ge (C-1)d_i^j.$ So, taking $C\ge 6,$ we obtain $\text{dist}(y, (B_i^j)^c)\ge \alpha t\wedge m_\beta (y),$ i.e., $(y,t)\in T_\gamma^{\alpha, \beta}(B_i^j).$ Observe that $B_i^j\in \mathscr{B}_{C 2^{2p+2}\sqrt{n}}: =\mathscr{B}_{\delta},$ for all $i,j.$  
Now, if $\mu$ is a Carleson measure, then 
\begin{align*}|\mu|&\{(y,t)\in \mathbb{R}^{n+1}_+: |f_j(y,t)|> \kappa\} \le |\mu|\left(\bigcup_{i \in {\mathbb N}}T_\gamma^{\alpha, \beta}(B_i^j)\right)\\ & \le \sum_{i\in {\mathbb N}}|\mu|(T_\gamma^{\alpha,\beta}(B_i^j)) \le \|\mu\|_C \sum_{i\in {\mathbb N}}\gamma(B_i^j)\\ &\lesssim_{n,\delta} \|\mu\|_C
\sum_{i\in {\mathbb N}}\gamma(Q_i^j) = \|\mu\|_C  \gamma(\Omega_\kappa^j). 
\end{align*}
So, 
\begin{align*}\iint_{\mathbb{R}^{n+1}_+}&|f(y,t)|\, d|\mu|(y,t)\le \sum_{j=1}^N\int_0^\infty |\mu|\{(y,t)\in \mathbb{R}^{n+1}_+: |f_j(y,t)|>\kappa\}\, d\kappa\\ &\lesssim_{n, \delta} \|\mu\|_C \sum_{j=1}^N \int_0^\infty 
\gamma\{x\in \mathbb{R}^n: S_{\infty,\alpha,\beta}(f_j)(x)>\kappa\}\, d\kappa\\ &\lesssim_{n,\delta, N} \|\mu\|_C\ \|f\|_{T_{\alpha,\beta}^{1,\infty}(\gamma)}.
\end{align*}}}
Thus, each Carleson measure induces a continuous linear functional on $T^{1,\infty}_{\alpha, \beta}(\gamma).$

Now, let us prove the converse. Given $l$ a continuous linear functional on $T^{1,\infty}_{\alpha, \beta}(\gamma)$. Then $l$ restricted to $\mathscr{C}_0(\mathbb{R}^{n+1}_+)$ is a continuous linear functional on $\mathscr{C}_0(\mathbb{R}^{n+1}_+)$ with the sup norm in $\mathbb{R}^{n+1}_+,$ since for $f\in \mathscr{C}_0(\mathbb{R}^{n+1}_+),$ $\|f\|_{T^{\alpha, \beta}_{1,\infty}(\gamma)}\le C_n \|f\|_{L^\infty(\mathbb{R}^{n+1}_+)}.$
From the Riesz representation theorem \cite[Theorem 6.19]{ru}, \cite{ar} there exists a unique finite signed Borel measure $\mu$ such that
\begin{equation}\label{representacionRiesz}
l(f)=\iint_{\mathbb{R}^{n+1}_+} f(y,t)\, d\mu(y,t), \ \ \ \ \hbox{ for any } f\in \mathscr{C}_0(\mathbb{R}^{n+1}_+),
\end{equation}
with $\|l\|=|\mu|(\mathbb{R}^{n+1}_+).$ 
Let us prove that $\mu$ is a Carleson measure. From the Riesz representation theorem, it is known that $\mu=\mu_+-\mu_-$ where $\mu_+$ is the Radon measure associated with the positive linear functional $l_1:=l\vee 0$, and $\mu_-$ is the one associated with the positive functional $l_2:=(l\vee 0)-l$ and such that $l$ is the difference between these two positive functionals.

Now, given $\{K_m\}_{m\in {\mathbb N}}$ an increasing sequence of compact sets that exhausts $\mathbb{R}^{n+1}_+,$ take $B\in \mathscr{B}_\delta,$ 
and $\varphi\in \mathscr{C}_0(\mathbb{R}^{n+1}_+)$ with $0\le \varphi\le 1,$ $\varphi=1$ in $T^{\alpha,\beta}_\gamma(B)\cap K_m$ and supported in a small neighborhood of $T^{\alpha,\beta}_\gamma(B)\cap K_m,$ contained in $T^{\alpha,\beta}_{\gamma}(2B),$ then we have 
\begin{align}\label{acotaciontienda}
\iint_{T^{\alpha, \beta}_\gamma(B)\cap K_m}\, d|\mu|(y,t)&\le \iint_{\mathbb{R}^{n+1}_+} \varphi (y,t)\, d|\mu|(y,t)\\ &=l_1(\varphi)+l_2(\varphi)\nonumber\\ &= |l(\varphi)| \le \|l\|\ \|\varphi\|_{T^{1,\infty}_{\alpha, \beta}(\gamma)}.\nonumber
\end{align}
It remains to prove that $\|\varphi\|_{T^{1,\infty}_{\alpha, \beta}(\gamma)}\lesssim \gamma(B).$ According with the definition of $\varphi,$ $\text{supp}(S_{\infty, \alpha, \beta}\ \varphi)\subset 2 B.$ Hence, $$\|\varphi\|_{T^{1,\infty}_{\alpha,\beta}(\gamma)} =\int_{\mathbb{R}^{n}}S_{\infty, \alpha, \beta}(\varphi)(x)\, d\gamma(x) \le C_n \|\varphi\|_\infty \gamma(2 B)\lesssim_{n,\delta} \gamma(B).$$ 
Thus, letting $m\to \infty$ in (\ref{acotaciontienda}), we obtain that $|\mu|(T^{\alpha, \beta}_{\gamma}(B))\lesssim_{n,\delta} \|l\|\ \gamma(B),$ for every $B\in \mathscr{B}_\delta.$ 
Notice that the right side of (\ref{representacionRiesz}) is also continuous on $T^{1,\infty}_{\alpha,\beta}(\gamma)$ taking into account inequality (\ref{acotacionmu}), and both sides of (\ref{representacionRiesz}) coincide on $\mathscr{C}_0(\mathbb{R}^{n+1}_+),$ a dense subspace of $T_{\alpha, \beta}^{1,\infty}(\gamma),$ so that formula in  (\ref{representacionRiesz}) is also valid for every $f\in T^{1,\infty}_{\alpha, \beta}(\gamma).$
 
\subsection{Proof of Theorem \ref{car2}}\label{car2proof}

\noindent{\bf Case $p=1,$ $1<q<\infty.$ }
For $f\in T^{1,q}_{\alpha, \beta}(\gamma)$ and $g\in T^{\infty, q'}_{\alpha,\beta} (\gamma),$ the following inequality holds.
\begin{equation}\label{dual1inftygral}
\iint_{\mathbb{R}^{n+1}_+}|f(y,t)|\, |g(y,t)|\ d\gamma(y)\frac{dt}{t}\le C \int_{\mathbb{R}^n}S_{q,\alpha,\beta} (f)(x)\, C_{q', \alpha,\beta}(g)(x)\, d\gamma(x).
\end{equation}
where $q'$ is the dual exponent to $q,$ and $$C_{q',\alpha,\beta} (g)(x)=\sup_{B: x\in B(c_B,  \alpha r_B\wedge m_\beta(c_B))}\left(\frac{1}{\gamma(B)}\iint_{T^{\alpha,\beta}_\gamma(B)}|g(y,t)|^{q'}\, d\gamma(y)\frac{dt}{t}\right)^{1/q'}.$$}

The idea of the proof is an adaptation to the Gaussian context of the proof in \cite[p. 313-314]{cms}.
{We define the truncated Gaussian cone $$\Gamma_\gamma^{\alpha, \beta, h}(x)=\{(y,t)\in \Gamma^{\alpha,\beta}_\gamma(x): t<h\}$$ and set
$$S_{q,\alpha,\beta,h} (f)(x)=\left(\iint_{\Gamma^{\alpha, \beta, h}_\gamma(x)}\frac{|f(y,t)|^q}{\gamma(B(y,\alpha t\wedge m_\beta(y)))}\, d\gamma (y)\frac{dt}{t}\right)^{1/q}.$$
It is immediate to see that $S_{q,\alpha, \beta,h}(f)(x)$ increases monotonically to   $S_{q,\alpha, \beta }\, f(x).$ Now, for every measurable function $g$ in $\mathbb{R}^{n+1}_+$, we define the {\bf stopping-time} $h(x)$ as
$$h(x)=\sup\{h>0: S_{q',\alpha, \beta,h}\, g(x)\le M C_{q',\alpha, \beta}\, g (x)\},$$
where the constant $M$ will be chosen later and will depend only on dimension $n$. In order to prove the inequality (\ref{dual1inftygral}) there is a key observation that involves the stopping time $h(x)$ that says
\begin{remark}\label{porcion}
There exists a constant $\lambda=\lambda_M\in (0,1)$ such that for every open ball $B=B(c_B,r_B)$ in $\mathbb{R}^n,$ we have \begin{equation}\label{puntosdensidad}
\gamma(\{x\in B(c_B, \alpha r_B\wedge m_\beta(c_B)): h(x)\ge r_B\})\ge \lambda\ \gamma (B(c_B, \alpha r_B\wedge m_\beta(c_B))).
\end{equation}
\end{remark}
Using this observation, Fubini's theorem followed by H\"older's inequality, we have
\begin{align*}
\iint_{\mathbb{R}^{n+1}_+}|f(y,t)|\, &|g(y,t)|\ d\gamma(y)\frac{dt}{t} \le \frac{1}{\lambda}\int_{\mathbb{R}^n}\iint_{\mathbb{R}^{n+1}_+}\frac{|f(y,t)|\, |g(y,t)|}{\gamma(B(y, \alpha t\wedge m_\beta (y)))}\\ & \qquad \int_{\mathbb{R}^n}\mathcal{X}_{\Gamma_\gamma^{\alpha, \beta, h(x)}(x)}(y,t)\, d\gamma(x) \ d\gamma(y)\frac{dt}{t}\\ &=\frac{1}{\lambda}\int_{\mathbb{R}^n}\iint_{\Gamma_\gamma^{\alpha, \beta, h(x)}(x)}\frac{|f(y,t)|\, |g(y,t)|}{\gamma(B(y,\alpha t\wedge m_\beta (y)))}\, d\gamma(y)\frac{dt}{t}\,  d\gamma(x)\\ &\le \frac{1}{\lambda}\int_{\mathbb{R}^n} S_{q,\alpha, \beta, h(x)} f (x)\, S_{q',\alpha, \beta,  h(x)} g (x)\, d\gamma(x)\\ & \le \frac{M}{\lambda}\int_{\mathbb{R}^n} S_{q,\alpha,\beta}\, f(x)\, C_{q',\alpha,\beta}\, g (x)\, d\gamma(x).
\end{align*}
This ends the proof of inequality (\ref{dual1inftygral}).
Now, we turn to the proof of inequality (\ref{puntosdensidad}).
Let us call $\widetilde{B}(c_B, \alpha r_B\wedge m_\beta(c_B)):=\{x\in B(c_B,\alpha r_B\wedge m_\beta (c_B)): h(x)\ge r_B\}.$ Note that $$x\in \widetilde{B}(c_B,\alpha r_B\wedge m_\beta(c_B)) \ \ \text{if\ and\ only\ if}\ \ S_{q', \alpha, \beta,r_B}\, g(x)\le M\, C_{q',\alpha, \beta}\, g (x).$$ Set $\widetilde{B}(c_B, \alpha r_B\wedge m_\beta(c_B))^c:=B(c_B,\alpha r_B\wedge m\beta (c_B))\setminus \widetilde{B}(c_B, \alpha r_B\wedge m\beta (c_B)).$ Take $g\in T^{\infty, q'}_{\alpha,\beta}(\gamma),$ with $g\not \equiv 0,$ then 
\begin{align*}
&\inf_{x\in B(c_B, \alpha r_B\wedge m_\beta(c_B))}  (M C_{q',\alpha,\beta}\, g (x))^{q'}  \gamma (\widetilde{B}(c_B,\alpha r_B\wedge m_\beta (c_B))^c) \\ 
& \le \int_{B(c_B, \alpha r_B\wedge m_\beta (c_B))} (S_{q',\alpha,\beta, r_B} g (x))^{q'}, d\gamma (x)\\ & =\int_{B(c_B, \alpha r_B\wedge m_\beta (c_B))} \iint_{\Gamma_\gamma^{\alpha, \beta, r_B}(x)}  \frac{|g(y,t)|^{q'}}{\gamma (B(y,\alpha t\wedge m_\beta (y)))}\, d\gamma(y)\frac{dt}{t}\, d\gamma(x)\\ & \le \iint_{R^{\alpha, \beta, r_B}_{\gamma}(B(c_B, \alpha r_B\wedge m_{\beta} (c_B)))} |g(y,t)|^{q'} \, d\gamma (y)\frac{dt}{t}\\ &\le \iint_{T_\gamma((2(\beta+1)^2+1) B(c_B, \alpha r_B\wedge m_\beta (c_B)))} |g(y,t)|^{q'}\, d\gamma(y)\frac{dt}{t}\\ & \le \gamma ((2(\beta+1)^2+1) B(c_B, \alpha r_B\wedge m_\beta(c_B)))\inf_{x\in B(c_B, \alpha r_B\wedge m_\beta (c_B))} (C_{q',\alpha, \beta}\, g(x))^{q'},
\end{align*}
namely $R_\gamma^{\alpha, \beta, r_B}(B(c_B, \alpha r_B\wedge m_{\beta} (c_B)))=\bigcup_{x\in B(c_B, \alpha r_B\wedge m_\beta (c_B))}\Gamma_\gamma^{\alpha, \beta, r_B}(x).$

Since $g$ is a nonzero $T^{\infty,q'}_{\alpha,\beta} (\gamma)$-function, then $\inf_{x\in B(c_B, \alpha r_B\wedge m_\beta (c_B))} C_{q',\alpha,\beta}\, g(x)\in (0,\infty),$ and we get that 
\begin{align*}
M^{q'} \gamma(\widetilde{B}(c_B, \alpha r_B\wedge m_\beta (c_B))^c)&\le \gamma((2(\beta+1)^2+1)B(c_B, \alpha r_B\wedge m_\beta (c_B))) \\ & \le K_\beta \gamma (B(c_B, \alpha r_B \wedge m_\beta (c_B))),
\end{align*}
thus $\gamma (\widetilde{B}(c_B, \alpha r_B\wedge m_\beta (c_B))^c)\le \frac{K_\beta}{M^{q'}}\gamma (B(c_B, \alpha r_B\wedge m_\beta (c_B))).$ Hence we pick $M$ large enough so that $K_\beta /M^{q'}\in (0,1),$ and take $\lambda= 1-\frac{K}{M^{q'}}.$ 
It remains to prove that $$R_{\gamma}^{\alpha, \beta, r_B}(B(c_B, \alpha r_B \wedge m_\beta (c_B)))\subseteq T_\gamma ((2(\beta+1)^2+1)B(C_B, \alpha r_B\wedge m_\beta (c_B))).$$ Let $(y,t)\in R_{\gamma}^{\alpha, \beta, r_B}(B(c_B, \alpha r_B\wedge m_\beta (c_B)))$ be given. Then there exists $x\in B(c_B, \alpha r_B \wedge m_\beta (c_B))$ such that $|y-x|<\alpha t\wedge m_\beta (y)$ and $t<r_B.$ Therefore \begin{align*}|y-c_B|\le |y-x|+|x-c_B|& <\alpha t\wedge m_\beta (y)+ \alpha r_B\wedge m_\beta (c_B)\\ &=2(\alpha t\wedge m_\beta(y))+\alpha r_B\wedge m_\beta(c_B)-(\alpha t\wedge m_\beta (y))\\ &\le 2(\alpha r_B\wedge m_\beta (y))+ \alpha r_B\wedge m_\beta (c_B)- (\alpha t\wedge m_\beta (y)).
\end{align*}
However, taking into account the relationship between $x$ and $y,$ and $x$ and $c_B,$ together with Lemma \ref{comparacion} we find that $m_\beta (y)\le (\beta+1)^2 m_\beta (c_B).$ Thus $|y-c_B|< (2(\beta+1)^2+1) (\alpha r_B\wedge m_\beta (c_B))-(\alpha t\wedge m_\beta (y))$, that is, $(y,t)\in T_\gamma^{\alpha, \beta} ((2(\beta+1)^2+1)B(c_B, \alpha r_B\wedge m_\beta (c_B))).$ }

Now, we will prove that the dual of $T^{1,q}_{\alpha,\beta}(\gamma)$ is isomorphic to $T^{\infty, q'}_{\alpha, \beta}(\gamma)$ with equivalent norms. An immediate consequence of the inequality (\ref{dual1inftygral}) is that every $g\in T^{\infty, q'}_{\alpha, \beta}(\gamma)$ induces a bounded linear functional on $T^{1,q}_{\alpha, \beta}(\gamma)$ by $f\to \iint_{\mathbb{R}^{n+1}_+}f(y,t)\, g(y,t)\ d\gamma(y)\frac{dt}{t}$. Let us prove the converse. Let $l$ be a bounded linear functional on $T^{1,q}_{\alpha, \beta}(\gamma).$ Let $K$ be a compact subset of $\mathbb{R}^{n+1}_+$, and $f$  a measurable function supported in $K,$ with $f\in L^q(K, d\gamma\frac{dt}{t}),$ then according to the proof of Lemma \ref{density}, $f\in T^{1,q}_{\alpha, \beta}(\gamma),$ with $\|f\|_{T^{1,q}_{\alpha, \beta}(\gamma)}\le C_K \|f\|_{L^q(K, d\gamma\frac{dt}{t})}.$ Thus $l$ induces a bounded linear functional on $L^q(K,d\gamma\frac{dt}{t}),$ and is thus representable by a function $g_K\in L^{q'}(K, d\gamma\frac{dt}{t}).$ Taking $\{K_m\}_{m\in {\mathbb N}}$ an increasing sequence of compact sets which exhausts $\mathbb{R}^{n+1}_+,$ we can construct a measurable function $g$ on $\mathbb{R}^{n+1}_+$ such that $g/K_m = g_{K_m},$ and so that $l(f)=\iint_{\mathbb{R}^{n+1}_+}f(y,t)\, g(y,t)\ d\gamma(y)\frac{dt}{t},$ whenever $f\in T^{1,q}_{\alpha, \beta}(\gamma),$ and $f$ has compact support $K.$ From the proof of Lemma \ref{density}, 
$f\in L^{q}(K,d\gamma\frac{dt}{t})$. Let us prove that $g\in T^{\infty, q'}_{\alpha, \beta}(\gamma).$ In fact, for any open ball $B,$ we calculate $\|g\mathcal{X}_{T^{\alpha, \beta}_\gamma (B)}\|_{L^{q'}(\mathbb{R}^{n+1}_+,d\gamma\frac{dt}{t})}$ through duality, that is, for $\phi\in L^{q}(\mathbb{R}^{n+1}_+,d\gamma\frac{dt}{t})$ with $\|\phi\|_{L^q(\mathbb{R}^{n+1}_+, d\gamma \frac{dt}{t})}\le 1,$ then  \begin{align*}
\left| \iint_{K_m} \mathcal{X}_{T^{\alpha, \beta}_\gamma(B)}(y,t)\phi(y,t)\, g(y,t)\ d\gamma(y)\frac{dt}{t}\right|&= |l(\mathcal{X}_{T^{\alpha,\beta}_\gamma (B)\cap K_m} \phi)|\\ &\le \|l\|\ \|\mathcal{X}_{T^{\alpha, \beta}_\gamma (B)} \phi\|_{T^{1,q}_{\alpha, \beta}(\gamma)}. 
\end{align*}
Now, 
\begin{align*}\|\mathcal{X}_{T^{\alpha, \beta}_\gamma(B)}\phi\|_{T^{1,q}_{\alpha, \beta}(\gamma)}&=\int_{B}S_{q,\alpha,\beta}(\mathcal{X}_{T^{\alpha, \beta}_\gamma (B)}\phi)(x)\, d\gamma(x) \\ & \le \gamma(B)^{1/q'}\left(\int_{\mathbb{R}^n} (S_{q,\alpha, \beta}(\mathcal{X}_{T^{\alpha, \beta}_\gamma(B)} \phi))(x))^q\, d\gamma(x)\right)^{1/q}\\ & = \gamma(B)^{1/q'} \|\phi\|_{L^{q}(\mathbb{R}^{n+1}_+, d\gamma\frac{dt}{t})}\le \gamma(B)^{1/q'}.
\end{align*}
Thus, \begin{align*}
\left|\iint_{T^{\alpha, \beta}_\gamma(B)}g \phi\ d\gamma(y)\frac{dt}{t}\right|&=\limsup_{m\to \infty}\left|\iint_{K_m} g \mathcal{X}_{T^{\alpha, \beta}_\gamma(B)} \phi\ d\gamma(y)\frac{dt}{t}\right|\\ & \le \|l\|\gamma(B)^{1/q'}.
\end{align*}
Therefore, $$C_{q', \alpha, \beta}\ g(x)=\sup_{B:\ x\in B(c_B, \alpha r_B\wedge m_\beta (c_B))}\frac{1}{\gamma(B)^{1/q'}}\|g\mathcal{X}_{T^{\alpha,\beta}_\gamma(B)}\|_{L^{q'}(\mathbb{R}^{n+1}_+, d\gamma \frac{dt}{t})}\le \|l\|,$$ and $\|g\|_{T^{\infty,q'}_{\alpha, \beta}(\gamma)}\le \|l\|.$ This representation of $l$ is then extendable to the whole space $T^{1,q}_{\alpha, \beta}(\gamma),$ since the subspace $ L^q_{c}(\mathbb{R}^{n+1}_+, d\gamma\frac{dt}{t})$ is dense in $T^{1,q}_{\alpha, \beta}(\gamma)$ according to the Lemma \ref{density}.
On the other hand, we get $\|l\|\lesssim \|g\|_{T^{\infty,q'}_{\alpha, \beta}(\gamma)}$ from (\ref{dual1inftygral}).

\noindent {\bf Case $1<p,q<\infty.$} Let us analyze the simple case $p=q.$ In this situation, it can be seen immediately that $T^{p,p}_{\alpha, \beta}(\gamma)=L^{p}(\mathbb{R}^n. d\gamma\frac{dt}{t})$ and therefore its dual is isometrically isomorphic to $L^{p'}(\mathbb{R}^{n+1}_+, d\gamma\frac{dt}{t})=T^{p',p'}_{\alpha,\beta}(\gamma).$

For $p\ne q.$ Take $f\in T^{p,q}_{\alpha, \beta}(\gamma)$ and $g\in T^{p',q'}_{\alpha, \beta}(\gamma),$ then
\begin{equation}\label{dualidadpq}
\iint_{\mathbb{R}^{n+1}_+}|f(y,t)|\, |g(y,t)|\, d\gamma(y)\frac{dt}{t}\le \int_{\mathbb{R}^n}S_{q,\alpha, \beta}f(x)\, S_{q',\alpha, \beta}g(x)\, d\gamma(x).
\end{equation}
In fact, applying an average followed by a H\"older's inequality with $q$ and $q'$, we obtain the following.
\begin{align*}
\iint_{\mathbb{R}^{n+1}_+}|f(y,t)|\, |g(y,t)|\,& d\gamma(y)\frac{dt}{t}= \iint_{\mathbb{R}^{n+1}_+}|f(y,t)|\, |g(y,t)|\, \\ &\frac{1}{\gamma(B(y, \alpha t\wedge m_\beta(y))}\int_{\mathbb{R}^n}\mathcal{X}_{B(y, \alpha t\wedge m_\beta (y))}(x)d\gamma(x)\, d\gamma(y)\frac{dt}{t}\\ & =\int_{\mathbb{R}^n}\iint_{\Gamma^{\alpha, \beta}_\gamma(x)}\frac{|f(y,t)|\, |g(y,t)|}{\gamma(B(y, \alpha t \wedge m_\beta (y)))}\, d\gamma(y)\frac{dt}{t}\, d\gamma(x)\\ &\le \int_{\mathbb{R}^n}S_{q, \alpha, \beta}f(x)\, S_{q',\alpha, \beta}g(x)\, d\gamma(x).
\end{align*}
So (\ref{dualidadpq}) shows, again using H\"older's inequality with $p$ and $p'$ that every element $g\in T^{p',q'}_{\alpha, \beta}(\gamma)$ induces a bounded linear functional on $T^{p,q}_{\alpha, \beta}(\gamma).$

Let us prove the converse.  Take $l$ a continuous linear functional on $T^{p,q}_{\alpha,\beta}(\gamma).$ Now, we argue as we have done for the case $p=1,$ $1<q<\infty.$ There exists a measurable function $g,$ that is locally in $L^{q'}(\mathbb{R}^{n+1}_+,d\gamma \frac{dt}{t}),$ so that \begin{equation}\label{formularepresentacion}
l(f)=\iint_{\mathbb{R}^{n+1}_+}f(y,t)\, g(y,t)\ d\gamma(y)\frac{dt}{t},
\end{equation}
whenever $f\in T^{p,q}_{\alpha, \beta}(\gamma)$ and $f$ have compact support in $\mathbb{R}^{n+1}_+.$ 

First, we assume that $1<p<q,$ then $q'<p'.$
Let $K$ be an arbitrary compact subset of $\mathbb{R}^{n+1}_+,$ and set $g_K=g\mathcal{X}_K.$ 
To prove that $g\in T^{p',q'}_{\alpha, \beta}(\gamma),$ it suffices to obtain the following boundedness with constant independent of $K$: \begin{equation}\label{acotacionTdual}
\|S_{q',\alpha,\beta}(g_K)\|_{L^{p'}(\gamma)}\le C \|l\|.
\end{equation}
 Let $s$ be the dual exponent to $\frac{p'}{q'}>1.$ 
 Observe that $p'=s' q'$ and
 \begin{align*}
\|g_K\|_{T^{p',q'}_{\alpha,\beta}(\gamma)}^{q'}&=\left[\left(\int_{\mathbb{R}^n}(S_{q',\alpha,\beta}(g_K)(x))^{p'} \, d\gamma(x)\right)^{1/{p'}}\right]^{q'}\\ & = \left(\int_{\mathbb{R}^n}(S_{q',\alpha, \beta}(g_K)(x))^{q's'}\, d\gamma(x)\right)^{1/{s'}}\\ &=\|S_{q',\alpha,\beta}(g_K)^{q'}\|_{L^{s'}(\gamma)}.
 \end{align*}
 Now, let us calculate the last norm on $L^{s'}(\gamma)$ through its dual, that is, 
 \begin{align*}
&\|S_{q',\alpha,\beta}(g_K)^{q'}\|_{L^{s'}(\gamma)}   = \underset{\|\phi\|_{L^{s}(\gamma)}\le 1}{\sup}\int_{\mathbb{R}^n}(S_{q',\alpha, \beta}(g_K)(x))^{q'}\, \phi(x)\, d\gamma(x)\\ &= \underset{\|\phi\|_{L^{s}(\gamma)}\le 1}{\sup}\iint_{\mathbb{R}^{n+1}_+}|g_k(y,t)|^{q'}\, \\ & \qquad \qquad \ \ \ \frac{1}{\gamma(B(y,\alpha t\wedge m_\beta(y)))}\int_{B(y, \alpha t\wedge m_\beta(y))} \phi(x)\, d\gamma(x)\, d\gamma(y)\frac{dt}{t}\\ &=: \underset{\|\phi\|_{L^{s}(\gamma)}\le 1}{\sup}\iint_{\mathbb{R}^{n+1}}|g_K(y,t)|^{q'}\, M_{\alpha, \beta}(\phi)(y,t)\, d\gamma(y)\frac{dt}{t}\\ &= \underset{\|\phi\|_{L^{s}(\gamma)}\le 1}{\sup} \iint_{\mathbb{R}^{n+1}_+} \mathrm{sgn} (g(y,t)) |g_K(y,t)|^{q'-1} M_{\alpha,\beta} (\phi)(y,t)\, g(y,t)\, d\gamma(y)\frac{dt}{t}\\ &= \underset{\|\phi\|_{L^{s}(\gamma)}\le 1}{\sup} \iint_{\mathbb{R}^{n+1}_+} f_\phi(y,t)\, g(y,t)\, d\gamma(y)\frac{dt}{t},
 \end{align*}
 where $f_\phi(y,t)=\mathrm{sgn}(g(y,t)) |g_K(y,t)|^{q'-1}M_{\alpha,\beta}\phi(y,t),$ with $\text{supp}(f_\phi)\subseteq K.$ Let us prove that $f_\phi\in T^{p,q}_{\alpha, \beta}(\gamma)$ showing that there exists a constant $C=C_{n,\beta}>0$ independent of $K$ and $\phi$ such that $\|f_\phi\|_{T^{p,q}_{\alpha, \beta}(\gamma)}\le C\, \|\phi\|_{L^s(\gamma)}\, \|g_K\|_{T^{p',q'}_{\alpha, \beta}(\gamma)}^{q'-1},$ and this last norm is finite, taking into account that $g_K\in L^{q'}_c(\mathbb{R}^{n+1}_+,d\gamma\frac{dt}{t})$ and the proof of the Lemma \ref{density}. 
In fact,
\begin{align*}
\|f_\phi\|_{T^{p,q}_{\alpha, \beta}(\gamma)} &=\left(\int_{\mathbb{R}^n}(S_{q,\alpha,\beta}(f_\phi)(x))^p\, d\gamma(x)\right)^{1/p}\\ &=\left(\int_{\mathbb{R}^n}\left(\iint_{\Gamma^{\alpha, \beta}_\gamma(x)}\frac{|g_K(y,t)|^{q(q'-1)}}{\gamma(B(y, \alpha t\wedge m_\beta(y)))}\, \right. \right.\\ &\left. \qquad \left. (\mathcal{X}_{B(y,\alpha t\wedge m_\beta(y))}(x) |M_{\alpha, \beta}(\phi)(y,t)|)^q\, d\gamma(y)\frac{dt}{t}\right)^{p/q}\, d\gamma(x)\right)^{1/p},
\end{align*}
 and $\mathcal{X}_{B(y, \alpha t\wedge m_\beta(y))}(x)|M_{\alpha, \beta}(\phi)(y,t)|\le \mathcal{M}_\beta(\phi)(x)$. Since $\mathcal{M}_\beta$ is bounded on $L^s (\gamma),$
  \begin{align*}
\|f_\phi\|_{T^{p,q}_{\alpha, \beta}(\gamma)}&\le \left(\int_{\mathbb{R}^n}\mathcal({M}_\beta(\phi)(x))^p (S_{q',\alpha, \beta}(g_K)(x))^{(q'-1)p}\, d\gamma(x)\right)^{1/p}\\ &\le \|\mathcal{M}_\beta(\phi)\|_{L^s(\gamma)}\, \|S_{q',\alpha,\beta}(g_K)^{(q'-1)}\|_{L^r(\gamma)}\\
&\lesssim_{n,\beta}\, \|\phi\|_{L^s(\gamma)}\, \|S_{q',\alpha, \beta}(g_K)\|_{L^{p'}(\gamma)}^{q'-1},
 \end{align*}
 with $\frac{1}{s}+\frac{1}{r}=\frac{1}{p},$ $s'=\frac{p'}{q'}>1,$ and $r=\frac{p'}{q'-1}.$

 Bringing all things together, we conclude that
 \begin{align*}\|g_K\|^{q'}_{T^{p',q'}_{\alpha, \beta}(\gamma)} &\le \underset{\|\phi\|_{L^s(\gamma)}\le 1}{\sup}l(f_\phi)\le \|l\|
 \underset{\|\phi\|_{L^s(\gamma)}\le 1}{\sup}\|f_\phi\|_{T^{p',q'}_{\alpha, \beta}(\gamma)} \\
 &\lesssim_{n,\beta}\, \|l\|\, \|g_K\|_{T^{p',q'}_{\alpha, \beta}(\gamma)}^{q'-1}.\end{align*}
 Therefore, inequality (\ref{acotacionTdual}) holds for every compact subset $K$ of $\mathbb{R}^{n+1}_+,$ and by the Beppo-Levy theorem $g\in T^{p',q'}_{\alpha, \beta}(\gamma)$. On the other hand, the formula (\ref{formularepresentacion}) valid for every $f\in T^{p,q}_{\alpha, \beta}(\gamma)$ with $\text{supp}(f)$ compact in $\mathbb{R}^{n+1}_+$ can be extended to the whole $T^{p,q}_{\alpha, \beta}(\gamma)$ since the space of the former functions is dense in the latter one. So, the conclusion of this theorem for $1<p<q$ is fully proved.

 For the case $q<p<\infty,$ then $p'<q'$ and we are in the previous case where $p, q$ are replaced by $p', q'$ respectively. So, we know that the dual of $T^{p',q'}_{\alpha, \beta}(\gamma)$ is identified with the Banach space $T^{(p')', (q')'}_{\alpha, \beta}(\gamma)=T^{p,q}_{\alpha, \beta}(\gamma).$ Hence the dual of $T^{p,q}_{\alpha, \beta}$
 is identified with the double dual of $T^{p',q'}_{\alpha, \beta}(\gamma).$ Thus, it suffices to show that the space $T^{p',q'}_{\alpha, \beta}(\gamma)$ is reflexive. Instead of using $p', q'$, we shall use $p, q$ such that $1<p<q<\infty$ and we will prove that the space $T^{p,q}_{\alpha, \beta}(\gamma)$ is a reflexive space. Taking into account the Eberlein-\v{S}mulyan theorem \cite{yo},  
 it suffices to show that the unit ball of $T^{p,q}_{\alpha, \beta}(\gamma)$ is sequentially weakly-compact, that is, given a sequence $\{f_j\}_{j\in {\mathbb N}}$ with $\|f_j\|_{T^{p,q}_{\alpha,\beta}(\gamma)}\le 1,\ \ \hbox{ for any } j\in {\mathbb N},$ we are to find a subsequence $\{f_{j_k}\}_{k\in {\mathbb N}}$ such that $\{l(f_{j_k})\}_{k\in {\mathbb N}}$ is a Cauchy sequence in $\mathbb{R}$ for every $l$ continuous linear functional on $T^{p,q}_{\alpha, \beta}(\gamma).$ Let $l$ be such a functional, then there exists $g\in T^{p',q'}_{\alpha, \beta}(\gamma)$ such that 
 $$l(f)=\langle f, g\rangle:=\iint_{\mathbb{R}^{n+1}_+}f(y,t)\, g(y,t)\, d\gamma(y)\frac{dt}{t},\ \ \ \hbox{ for any } f\in T^{p,q}_{\alpha,\beta}(\gamma).$$ Fix $\{K_m\}_{m\in {\mathbb N}}$ an increasing sequence of compact subsets of $\mathbb{R}^{n+1}_+$ which exhausts $\mathbb{R}^{n+1}_+.$ Set $f_j^m:=f_j\mathcal{X}_{K_m}$, then, according to (\ref{mm}), 
 $\{f_j^m\}_{j\in {\mathbb N}}$ is a bounded sequence on $L^{q}(K_m,d\gamma\frac{dt}{t})$ a reflexive space, so there exists a subsequence of $\{f_j^m\}_{j\in {\mathbb N}}$ which is weakly convergent on $L^{q}(K_m, d\gamma\frac{dt}{t}).$ With standard arguments, we can select a subsequence $\{f_{j_k}\}_{k\in {\mathbb N}}$ of $\{f_j\}_{j\in {\mathbb N}}$ so that it is weakly convergent on $L^q(K,d\gamma\frac{dt}{t})$, for each compact set $K\subset \mathbb{R}^{n+1}_+.$ Now, 
 define $g_m:=g\mathcal{X}_{K_m}.$ Given $\epsilon>0,$ and taking into account the proof of Lemma \ref{density}, there exists $m\in {\mathbb N}$ such that $\|g-g_m\|_{T^{p',q'}_{\alpha, \beta}(\gamma)}<\epsilon$ with $g_m\in L^{q}(K_m,d\gamma\frac{dt}{t}).$ Thus,
 \begin{align*}
|l(f_{j_i})-l(f_{j_k})|&\le |\langle f_{j_i}-f_{j_k}, g-g_m\rangle |+\langle f_{j_i}-f_{j_k}, g_m\rangle|\\ &\le \|f_{j_i}-f_{j_k}\|_{T^{p,q}_{\alpha, \beta}(\gamma)}\, \|g-g_m\|_{T^{p',q'}_{\alpha, \beta}(\gamma)}+ |\langle f_{j_i}, g_m\rangle - \langle f_{j_k}, g_m\rangle |\\ &< 2\epsilon + 
|\langle f_{j_i}, g_m\rangle - \langle f_{j_k}, g_m\rangle |.
 \end{align*}
Since $\{f_{j_k}\}_{k\in {\mathbb N}}$ is weakly convergent on $L^{q}(K_m,d\gamma\frac{dt}{t}),$ $\langle f_{j_i}-f_{j_k},g_m\rangle \to 0,$ as $i,k\to \infty,$ this proves the weak convergence and concludes the proof of the theorem.

 \section{Proof of Section \ref{new}}\label{independence}

 \subsection{Proof of Theorem \ref{independencia}}\label{car3proof}
 First, a remark that will be used in the proof of Theorem \ref{independencia}.
 \begin{remark}
For every $B\in\mathscr{B}_\beta,$ \begin{equation}\label{calculobola}
e^{-(2+\beta)\beta} \, e^{-|c_B|^2} \, |B|\le \gamma(B)\le e^{(2+\beta)\beta}\, e^{-|c_B|^2}\, |B|,
\end{equation}
where $|B|$ means the Lebesgue measure of $B,$ i.e., $|B|=\omega_n\, r_B^n$ being $\omega_n$ the Lebesgue measure of the unit ball. So, inequality (\ref{calculobola}) will be denoted by $\gamma (B)\sim_{n,\beta}\, e^{-|c_B|^2}\, r_B^n.$
\end{remark}
 \begin{proof}
 It is immediate for $p=q$, since $$\|f\|_{T^{p,p}_{\alpha, \beta}(\gamma)}=\left(\iint_{\mathbb{R}^{n+1}_+}|f(y,t)|^p\, d\gamma(y)\frac{dt}{t}\right)^{1/p}=\|f\|_{T^{p,p}_{\delta,\lambda}(\gamma)}.$$ So, we assume $p\ne q.$
 
First we fix $\beta>0$, it suffices to prove the following one-way inequality. \begin{equation}\label{unsentido}\|f\|_{T^{p,q}_{\alpha, \beta}(\gamma)}\lesssim_{n, \alpha, \beta, \delta}\, \|f\|_{T^{p,q}_{\delta, \beta}(\gamma)},\ \ \ \ \hbox{ for any }\, \alpha, \delta >0.\end{equation}
{\bf Case $1\le p<\infty.$} 

Let us take $\alpha \le \delta$, then $\Gamma_\gamma^{\alpha, \beta}(x)\subseteq \Gamma_\gamma^{\delta, \beta}(x),$ and $\gamma(B(y,\alpha t\wedge m_\beta(y)))\sim_{n, \alpha, \beta, \delta} \gamma(B(y, \delta t\wedge m_\beta(y))).$ Observe that the inclusion is immediate and with respect to the equivalence, according to (\ref{calculobola}), it suffices to prove that their radii are equivalent. In fact, $\alpha t \wedge m_\beta(y)\le \delta t\wedge m_\beta(y)\le \frac{\delta}{\alpha}(\alpha t\wedge m_\beta(y)).$ From the above remarks we conclude that $S_{q,\alpha, \beta} f(x)\lesssim_{n,\alpha, \beta, \delta}\, S_{q,\delta, \beta} f(x)$ and so the inequality (\ref{unsentido}) follows.

Now, for $\delta < \alpha,$ we analyze two cases $1\le q<p$ and $1\le p<q.$
For $1\le q<p,$ then \begin{align*}
\|f\|_{T^{p,q}_{\alpha, \beta}(\gamma)}&= \|(S_{q,\alpha, \beta}f)^q\|_{L^{\frac{p}{q}}(\gamma)}^{\frac{1}{q}}\\ &=\sup_{
g\ge 0,\, g\in L^{\left(p/q\right)'}(\gamma),\, \|g\|_{\left(p/q\right)', \gamma}\le 1}\left(\int_{\mathbb{R}^n}\, (S_{q,\alpha, \beta}\, f(x))^q\, g(x)\, d\gamma(x)\right)^{1/q}.
\end{align*}
In what follows, the well-known functions $$M_{\alpha, \beta}\, g(y,t)=\frac{1}{\gamma(B(y, \alpha t\wedge m_\beta (y)))}\int_{B(y,\alpha t \wedge m_\beta(y))}\, g(x)\, d\gamma$$ and $$\displaystyle \mathcal{M}_\beta\, g (x)=\sup_{B\in \mathscr{B}_\beta, x\in B}\frac{1}{\gamma(B)}\int_B\, |g(y)|\, d\gamma(y)$$ will be used. On the other hand, observe also that $B(y,\delta t \wedge m_\beta (y))\subset B(y, \alpha t\wedge m_\beta (y))$ and $B(y, \alpha t\wedge m_\beta (y))\in \mathscr{B}_\beta.$
\begin{align*}
&\int_{\mathbb{R}^n} (S_{q,\alpha,\beta}\,  f(x))^q\, g(x)\, d\gamma(x) = \iint_{\mathbb{R}^{n+1}_+}|f(y,t)|^q\, M_{\alpha, \beta}\, g(y,t)\, d\gamma(y)\frac{dt}{t}\\  &= \int_{\mathbb{R}^n}\iint_{\Gamma^{\delta, \beta}_\gamma(x)}\frac{|f(y,t)|^q}{\gamma(B(y, \delta t\wedge m_\beta(y)))}  \mathcal{X}_{B(y, \delta t\wedge m_\beta (y))}(x)\\ & \qquad \qquad M_{\alpha, \beta}\, g(y,t)\, d\gamma(y)\frac{dt}{t}\, d\gamma(x)\\ &\le \int_{\mathbb{R}^n}\, (S_{q,\delta, \beta}\, f(x))^q\, \mathcal{M}_\beta\, g(x)\, d\gamma(x)\\ &\le \|(S_{q,\delta, \beta}\, f)^q\|_{L^{p/q}(\gamma)} \, \|\mathcal{M}_\beta\, g\|_{L^{(p/q)'}(\gamma)}\lesssim_\beta \|f\|^q_{T^{p,q}_{\delta,\beta}(\gamma)}\, \|g\|_{L^{(p/q)'}(\gamma)}.
\end{align*}
Hence, $\|f\|_{T^{p,q}_{\alpha, \beta}(\gamma)}\lesssim_\beta \|f\|_{T^{p,q}_{\delta, \beta}(\gamma)}.$

For $1\le p<q,$ we will use a modified version of Lemma \ref{lema: desigualdad integral, relacion entre eta y beta}.

\begin{lemma}\label{reversefubini1}
There exists $\eta\in (0,1)$ such that for every measurable subset $F$ of $\mathbb R^n$ and every non-negative measurable function $H$ on $\mathbb R^{n+1}_+$, we have
\begin{align}\label{reversefubini}
    \iint_{R^{\alpha,\beta}_\gamma\left(\widetilde{F}_{\beta}^{[\eta]}\right)} & H(y,t) d\gamma(y) \frac{dt}{t}\nonumber \\
    & \lesssim_{n,\alpha, \beta, \delta, \eta} \int_F \left(\iint_{\Gamma_\gamma^{\delta, \beta}(x)} \frac{H(y,t)}{\gamma(B(y,\delta t\wedge m_\beta(y))} d\gamma(y) \frac{dt}{t}\right) d\gamma(x),
\end{align}
being $\widetilde{F}^{[\eta]}_\beta=\{x\in \mathbb{R}^n: \gamma(B\cap F)\ge \eta\, \gamma(B),\ \ \hbox{ for any } B\in \mathscr{B}_\beta\ \, \text{and}\ \, x\in B\}\subset F_\beta^{[\eta]}.$
\end{lemma}
\begin{proof}
Applying Fubini's theorem, we obtain
\begin{align*}
\int_F & \left(\iint_{\Gamma_\gamma^{\delta, \beta}(x)}\, \frac{H(y,t)}{\gamma(B(y,\delta t \wedge m_\beta (y)))}\right) d\gamma(x) \\& =\iint_{\mathbb{R}^{n+1}_+}\, \frac{\gamma(F\cap B(y, \delta t \wedge m_\beta (y)))}{\gamma(B(y, \delta t \wedge m_\beta (y)))}\, H(y,t)\, d\gamma(y)\frac{dt}{t}.
\end{align*}
Similarly to that we have done in the proof of Lemma \ref{lema: desigualdad integral, relacion entre eta y beta}, let $(y,t)\in R_\gamma^{\alpha, \beta}\left(\widetilde{F}^{[\eta]}_\beta\right)$, then there exists $x\in \widetilde{F}_\beta^{[\eta]}$ such that $|y-x|<\alpha t\wedge m_\beta(y),$ i.e., $x\in B(y,\alpha t\wedge m_\beta (y)),$ and $\gamma(F\cap B(y, \alpha t\wedge m_\beta (y)))\ge \eta \, \gamma(B(y, \alpha t\wedge m_\beta(y))).$  Taking into account (\ref{calculobola}),
\begin{align*}
\gamma(B(y, \delta t\wedge m_\beta(y))) &\ge e^{-(2+\beta)\beta} \omega_n (\delta t\wedge m_\beta(y))^n e^{-|y|^2}\\ & \ge \left(\frac{\delta}{\alpha}\right)^n e^{-(2+\beta)\beta} \, \gamma(B(y, \alpha t\wedge m_\beta (y)))\\
& := c_{n, \alpha, \beta, \delta}\, \gamma(B(y, \alpha t\wedge m_\beta(y))).
\end{align*}
Notice $c_{n, \alpha, \beta, \delta}<1.$
Thus, \begin{align*}\gamma(B(y, \alpha t\wedge m_\beta (y))\setminus B(y,\delta t\wedge m_\beta (y)))&\le (1-c_{n, \alpha, \beta, \delta})\gamma (B(y, \alpha t\wedge m_\beta (y)))\\ &  :=\widetilde{c}_{n, \alpha,\beta, \delta}\  \gamma(B(y,\alpha t\wedge m_\beta (y))).
\end{align*}
However, \begin{align*}
\gamma (F\cap B(y, \delta t\wedge m_\beta (y)))&\ge \gamma(F\cap B(y, \alpha t\wedge m_\beta (y))) \\
& \;\;\; - \gamma(B(y, \alpha t\wedge m_\beta (y))\setminus B(y,\delta t\wedge m_\beta (y)))\\ & \ge (\eta- \widetilde{c}_{\alpha, \beta, \delta })\, \gamma(B(y,\alpha t\wedge m_\beta(y)))\\ &\ge (\eta- \widetilde{c}_{\alpha, \beta, \delta })\, \gamma(B(y,\delta t\wedge m_\beta(y))).
\end{align*}
Hence, if $\eta$ is chosen sufficiently close to $1,$ then on $R_\gamma^{\alpha,\beta}\left(\widetilde{F}_\beta^{[\eta]}\right),$ the quotient $\frac{\gamma(F\cap B(y, \delta\wedge m_\beta (y)))}{\gamma(B(y,\delta t\wedge m_\beta (y)))}\ge \lambda$ for $\lambda=\eta-\widetilde{c}_{n,\alpha,\beta, \delta}$ in $(0,1),$ and the inequality (\ref{reversefubini}) holds.
\end{proof}
We now turn to the proof of (\ref{unsentido}). Fix $\lambda >0,$ and let $F_\lambda=\{x\in \mathbb{R}^n: S_{q, \delta, \beta}f(x)\le \lambda\},$ $O_\lambda=\mathbb{R}^n\setminus F_\lambda= \{x\in \mathbb{R}^n: S_{q, \delta, \beta}f(x)> \lambda\}.$ Now, we take $\eta$ from Lemma \ref{reversefubini1} and consider the closed set $(\widetilde{F}_\lambda)_\beta^{[\eta]}\subset F_\lambda,$ and define $O_{\lambda, \beta}^{[\eta]}:=\mathbb{R}^n\setminus (\widetilde{F}_\lambda)_\beta^{[\eta]}.$ We apply Fubini's theorem to obtain the following inequality.
$$\int_{(\widetilde{F}_\lambda)_\beta^{[\eta]}}(S_{q,\alpha, \beta}f(x))^q\, d\gamma(x)\le \iint_{R^{\alpha, \beta}_\gamma((\widetilde{F}_\lambda)_\beta^{[\eta]})}|f(y,t)|^q\, d\gamma(y)\frac{dt}{t}.$$
Next, we apply Lemma \ref{reversefubini1} with $H(y,t)=|f(y,t)|^q,$ and obtain,
$$\int_{(\widetilde{F}_\lambda)_\beta^{[\eta]}}(S_{q,\alpha, \beta}f(x))^q\, d\gamma(x)\lesssim_{n,\alpha,\beta,\delta,\eta}\int_{F_\lambda} (S_{q,\delta, \beta}f(x))^q\, d\gamma(x).$$
It follows from this that
$$\gamma\{x\in \mathbb{R}^n: S_{q,\alpha, \beta}f(x)>\lambda\}\le \gamma(O_{\lambda, \beta}^{[\eta]})+\frac{C_{n,\alpha,\beta,\delta, \eta}}{\lambda^q}\int_{F_\lambda} (S_{q,\delta, \beta}f(x))^q\, d\gamma(x).$$
However, since $O_{\lambda, \beta}^{[\eta]}=\{x\in \mathbb{R}^n : \mathcal{M}_\beta (\mathcal{X}_{O_\lambda})(x)>\frac{1}{1-\eta}\}$, then  $\gamma(O_\beta^{[\eta]})\le C_{\beta, \eta}\, \gamma(O),$ for $\mathcal{M}_\beta$ is of weak-type $(1,1)$ with respect to $\gamma,$ see \cite{ur}. Together, this implies that
\begin{align}\label{desigualdadfunciondensidad}
\gamma(\{x\in \mathbb{R}^n: S_{q,\alpha, \beta}f(x)>\lambda\})&\lesssim \gamma(\{x\in \mathbb{R}^n: S_{q,\delta, \beta}f(x)>\lambda\})+ \nonumber\\&\qquad \frac{1}{\lambda^q}\int_{\{x\in \mathbb{R}^n: S_{q,\delta, \beta}f(x)\le \lambda\}}(S_{q,\delta,\beta}f(x))^q\, d\gamma(x).
\end{align}
If we multiply both sides of (\ref{desigualdadfunciondensidad}) by $\lambda^{p-1}$ and integrate we then get the inequality (\ref{unsentido}) for $1\le p<q.$

{\bf Case $q=\infty,$ $1< p<\infty.$}

For $\alpha\le \delta, $ we have $\Gamma_\gamma^{\alpha, \beta}(x)\subseteq \Gamma_{\gamma}^{\delta, \beta}(x)$, hence $S_{\infty,\alpha,\beta}\, f(x)\le S_{\infty,\delta, \beta}\, f(x)$ for all $x\in \mathbb{R}^n$, so inequality (\ref{unsentido}) holds.

For $\delta <\alpha$, we shall prove that \begin{equation}\label{desigualdadunsentido}S_{\infty,\alpha, \beta}\, f(x)\lesssim_{\alpha, \beta, \delta}\mathcal{M}_\beta (S_{\infty, \delta, \beta}\, f)(x)\end{equation}for all $x\in \mathbb{R}^n$, and so inequality (\ref{unsentido}) will follow also since the non-centered Gaussian Hardy-Littlewood maximal function on $\mathscr{B}_\beta$ is of $\gamma$-strong type $(p,p)$ for $1<p<\infty.$

Let $(y,t)\in \Gamma_\gamma^{\alpha,\beta}(x) $ be given. For all $z\in B(y,\delta t\wedge m_\eta (y))$, $(y,t)\in \Gamma_\gamma^{\delta, \beta}(z)$ and $|f(y,t)|\le S_{\infty, \delta, \beta}\, f(z).$ Thus,
\begin{align*}
|f(y,t)|&\le \inf_{z\in B(y,\delta t\wedge m_\beta (y))}\, S_{\infty, \delta, \beta}\, f(z)\\ &\le \frac{1}{\gamma(B(y,\delta t\wedge m_\beta (y)))}\int_{B(y, \delta t\wedge m_\beta (y))}\, S_{\infty, \delta, \beta}\, f(z)\, d\gamma(z).
\end{align*}
Taking into account that $\gamma(B(y,\delta t\wedge m_\beta (y)))\ge e^{-(\beta+2)\beta}\left(\frac{\delta}{\alpha}\right)^n\, \gamma(B(y, \alpha t\wedge m_\beta(y))),$ $\delta< \alpha$ and $x\in B(y,\alpha t\wedge m_\beta (y,))$ we get $|f(y,t)|\lesssim_{\alpha, \beta, \delta}\mathcal{M}_\beta (S_{\infty, \delta,\beta}\, f)(x),$ for all $(y,t)\in \Gamma_\gamma^{\alpha, \beta}(x),$ and hence inequality (\ref{desigualdadunsentido}) follows.

So far, we have proved that for $\beta>0$ fixed, we have $T_{\alpha, \beta}^{p,q}(\gamma)=T_{\delta, \beta}^{p,q}(\gamma)$ for all $\alpha, \delta>0.$ Now we fix $\alpha>0,$ and let $\beta$ vary, we shall obtain the equality $T^{p,q}_{\alpha, \beta}(\gamma)=T_{\alpha,\lambda}^{p,q}(\gamma)\ \hbox{ for any } \beta,\lambda>0$ if we apply the same procedure as before when the aperture of the Gaussian cones varies.
Therefore, $T^{p,q}_{\alpha, \beta}(\gamma)=T^{p,q}_{\delta, \lambda}(\gamma)$ for $1\le p,q<\infty$ and for $q=\infty$ and $1<p<\infty,$ for all $\alpha,\beta, \delta, \lambda>0.$ 

{\bf Case $p=\infty$ and $1<q<\infty$.} 

From Theorem \ref{car2}, $T^{\infty, q}_{\alpha,\beta}(\gamma)$ can be identified with the dual of $T^{1,q'}_{\alpha, \beta}(\gamma)$ and the same is true for $T^{\infty, q}_{\delta, \lambda}(\gamma)$ with the dual of $T^{1,q'}_{\delta, \lambda}(\gamma).$ But we have just proved that $T^{1,q'}_{\alpha,\beta}(\gamma)=T_{\delta, \lambda}^{1, q'}(\gamma),$ then $T^{\infty, q}_{\alpha,\beta}(\gamma)=T^{\infty, q}_{\delta, \lambda}(\gamma).$

That is, for each $p$ and $q$ there is only one Gaussian tent space that is symbolized by $T^{p,q}(\gamma).$
 \end{proof}

\end{document}